\documentclass[
final
]{dmtcs-episciences}

\usepackage[utf8]{inputenc}
\usepackage{subfigure}

\usepackage[round]{natbib}

\usepackage{tikz} 
\usepackage{hyperref}
\usepackage{eqnarray,amsthm, amssymb, amsmath,verbatim}
\usepackage{mathrsfs}

\usepackage{enumerate}
\usepackage{graphicx}
\usepackage{subfigure}
\usepackage{url}

\allowdisplaybreaks%[1]

\newtheorem{lem}{Lemma}
\newtheorem{theorem}{Theorem}

\newtheorem{corollary}{Corollary}

\newcommand{\tref}[1]{Theorem \ref{theorem:#1}}
\newcommand{\lref}[1]{Lemma \ref{lem:#1}}
\newcommand{\fref}[1]{Figure \ref{fig:#1}}

\newcommand{\cref}[1]{Conjecture \ref{conjecture:#1}}
\newcommand{\coref}[1]{Corollary \ref{corollary:#1}}

\newcommand{\la}{\lambda}

\newcommand{\sg}{\sigma}

\newcommand{\ga}{\gamma}
\newcommand{\Ga}{\Gamma}

\newcommand{\inv}{\mathrm{inv}}
\newcommand{\coinv}{\mathrm{coinv}}
\newcommand{\red}{\mathrm{red}}

\newcommand{\ret}{\mathrm{ret}}
\newcommand{\linv}{\mathrm{linv}}
\newcommand{\Sn}[1]{\mathcal{S}_{#1}}
\newcommand{\Snn}{\mathcal{S}_{n}}
\newcommand{\Dn}[1]{\mathcal{D}_{#1}}
\newcommand{\Dnn}{\mathcal{D}_{n}}

\newcommand{\LRmin}{\mathrm{LRmin}}

\newcommand{\fillll}[3]{\node at (#1-.5,#2-.5) {$#3$}}
\newcommand{\filllll}[2]{\node at (#1-.5,#2-.5) {$#2$}}
\newcommand{\fillcross}[2]{\draw[thick] (#1-1,#2-1)--(#1,#2);\draw[thick] (#1-1,#2)--(#1,#2-1)}
\newcommand{\fillgcross}[2]{\draw[thick,green!50!black] (#1-1,#2-1)--(#1,#2);\draw[thick,green!50!black] (#1-1,#2)--(#1,#2-1)}
\newcommand{\thn}[1]{$#1^\textnormal{th}$}

\newcommand{\occr}{\mathrm{occr}}
\newcommand{\coarea}{\mathrm{coarea}}
\newcommand{\area}{\mathrm{area}}	
\newcommand{\Abar}{\overline{A}}

\author{Dun Qiu\affiliationmark{1}%\thanks{}
  \and Jeffrey Remmel\affiliationmark{2}
}
\title[Classical patterns in $\mathcal{S}_{n}(132)$ and $\mathcal{S}_{n}(123)$]{Classical pattern distributions in\\ $\mathcal{S}_{n}(132)$ and $\mathcal{S}_{n}(123)$}
\affiliation{
  Department of Mathematics, Beijing Jiaotong University, Beijing, P. R. China\\
  Department of Mathematics, University of California San Diego, La Jolla, USA}
\keywords{permutation statistics, classical patterns, Catalan numbers, Dyck paths}
\received{2019-1-17}
\revised{2019-9-4}
\accepted{2019-10-15}
\begin{document}
\publicationdetails{21}{2019}{2}{4}{5088}
\maketitle
\begin{abstract}
Classical pattern avoidance and occurrence are well studied in the symmetric group $\mathcal{S}_{n}$. In this paper, we provide explicit recurrence relations to the generating functions counting the number of classical pattern occurrence in the set of 132-avoiding permutations and the set of 123-avoiding permutations.
\end{abstract}

\section{Introduction}
Let $\Snn$ denote the set of permutations of size $n$. Given a sequence $w = w_1 \cdots w_n$ of distinct integers, we
let $\red(w)$ be the permutation that we replace the
$i$-th smallest integer in $\sg$ with $i$.  For
example, $\red(4685) = 1342$.  Given a
permutation $\tau=\tau_1 \cdots \tau_j$ in $\Sn{j}$, we say that the pattern $\tau$ {\em occurs} in $\sg = \sg_1 \cdots \sg_n \in \Sn{n}$ if   there exist $1 \leq i_1 < \cdots < i_j \leq n$ such that 
$\red(\sg_{i_1} \cdots \sg_{i_j}) = \tau$.   We say 
that a permutation $\sg$ {\em avoids} the pattern $\tau$ if $\tau$ does not 
occur in $\sg$. In the theory of permutation patterns,  $\tau$ is called a {\em classical pattern}. 

We let $\Sn{n}(\tau)$ denote the set of permutations in $\Sn{n}$ 
which avoid $\tau$. If $\Gamma$ is a collection of permutations, then we let $\Sn{n}(\Gamma)$ denote the set of permutations in $\Sn{n}$ 
that avoid each permutation in $\Gamma$. We write $\occr_{\tau}(\sg)$ for the number of the pattern $\tau$ occurrences in $\sg$. For example, the permutation $\sg=867943251$ avoids the pattern $132$, while it contains the pattern $123$ and $\occr_{123}(\sg)=1$ since only the subsequence $6,7,9$ matches $123$. 

Classical patterns have been studied separately for a long
time. It is well known that for all $n \geq 1$, 
$|\Snn(132)|=|\Snn(123)|=C_n$, where $C_n = \frac{1}{n+1}\binom{2n}{n}$ is the 
$n^{\mathrm{th}}$ Catalan number.
\cite{M1,M2} enumerated the number of permutations in $\Snn$ avoiding 2 classical patterns. \cite{MV2,MV3} enumerated the number of permutations in $\Snn(132)$ or $\Snn(123)$ that has $0$ or $1$ occurrence of a certain pattern. See \cite{kit} for a comprehensive introduction to patterns in permutations.
However, there is not much research about the distribution of classical patterns in $\Snn(\tau)$. \cite{MV1} gave a continued fraction form of the generating function of the distribution of the pattern $12\cdots k$ in $\Snn(132)$.
Very recently, \cite{J1,J2} studied patterns in random permutations avoiding 132 or 123 in a probabilistic way.
\cite{PQR} studied consecutive pattern matches in $\mathcal{S}_n(132)$ and $\mathcal{S}_n(123)$.

Given two sets of permutations $\Lambda=\{\la_1,\ldots,\la_r\}$ and $\Gamma = \{\ga_1,\ldots,\ga_s\}$, it is natural to study the 
distribution of classical patterns $\ga_1,\ldots,\ga_s$ in $\Sn{n}(\Lambda)$. That is, we want to 
study generating functions of the form 
\begin{equation}
Q_{\Lambda}^{\Gamma}(t,x_1,\ldots,x_s) : = 1 + \sum_{n\geq 1} t^n  Q_{n,\Lambda}^{\Gamma}(x_1,\ldots,x_s),
\end{equation}
where
\begin{equation}
Q_{n,\Lambda}^{\Gamma}(x_1,\ldots,x_s) : = \sum_{\sg\in\Sn{n}(\Lambda)} x_1^{\occr_{\gamma_1}(\sg)}\cdots x_s^{\occr_{\gamma_s}(\sg)}.
\end{equation}
When $\Lambda=\{\la\}$ and $\Gamma=\{\ga\}$ are singletons, we write
\begin{equation}
Q_{\lambda}^{\gamma}(t,x) : = 1 + \sum_{n\geq 1} t^n  Q_{n,\lambda}^{\gamma}(x)\textnormal{ \ and \ }Q_{n,\lambda}^{\gamma}(x) : = \sum_{\sg\in\Sn{n}(\lambda)} x^{\occr_{\gamma}(\sg)}.
\end{equation}
The main goal of this paper is to study the distribution of 
classical patterns in the set of 132-avoiding permutations and the set of  123-avoiding permutations using a recursive method.

To study the generating functions $Q_\lambda^\gamma(t,x)$ when $\lambda$ is $132$ or $123$, we shall first study the symmetries in $\Sn{n}(132)$ and $\Sn{n}(123)$.
Given a permutation $\sg = \sg_1 \sg_2 \ldots \sg_n \in \Snn$, we let $\sg^r$ be the \emph{reverse} of $\sg$ defined by $\sg^r = \sg_n \ldots \sg_2 \sg_1$, and $\sg^c$ be the \emph{complement} of $\sg$ defined by $\sg^c = (n+1 -\sg_1) (n+1-\sg_2) \ldots (n+1 -\sg_n)$. Further, we let $\sg^{rc}=(\sg^r)^c$ be the \emph{reverse-complement} of $\sg$, and $\sg^{-1}$ be the \emph{inverse} of $\sg$. For example, for $\sg=15324$, we have $\sg^r=42351,\ \sg^c=51342,\ \sg^{rc}=24315$, and  $\sg^{-1}=14352$.

The actions above give several symmetries about classical pattern distributions in $\Snn(\la)$, and we have the following lemma.

\ 
\begin{lem}\label{lem:1}For any permutations $\la$ and $\gamma$, we have
\begin{equation*}
	Q_{\la}^\gamma(t,x)=Q_{\la^{*}}^{\gamma^{*}}(t,x),
\end{equation*}
	where $*$ is one of the actions $r,\ c,\ rc$ or $-1$.
\end{lem}
\begin{proof}
	We shall only prove the case when $*=r$. All the other cases are based on similar proofs.
	
	The action reverse bijectively sends a permutation $\sg\in\Snn(\la)$ to a permutation $\sg^r\in\Snn(\la^r)$, while each occurrence of $\ga$ in $\sg$ is sent to an occurrence of $\ga^r$ in $\sg^r$,  thus $\occr_{\ga^{r}}(\sg^{r})=\occr_{\ga}(\sg)$, and it follows immediately that $Q_{\la}^\gamma(t,x)=Q_{\la^{r}}^{\gamma^{r}}(t,x)$.
\end{proof}

When the permutation $\la$ is of length three, since $123=123^{rc}=123^{-1}$ and $132=132^{-1}$, we have the following corollary.

\ 
\begin{corollary}Given any permutation pattern $\gamma$, we have
\begin{equation*}
	Q_{123}^\gamma(t,x)=Q_{123}^{\gamma^{rc}}(t,x)=Q_{123}^{\gamma^{-1}}(t,x), \ \ \  Q_{132}^\gamma(t,x)=Q_{132}^{\gamma^{-1}}(t,x).
\end{equation*}
\end{corollary}

When $\gamma$ is a pattern of length three, we have the following corollary.

\ 
\begin{corollary}We have the following equations,
	\begin{equation*}
	Q_{132}^{231}(t,x)=Q_{132}^{312}(t,x),\qquad
	Q_{123}^{132}(t,x)=Q_{123}^{213}(t,x),\quad\mbox{and}\quad
	Q_{123}^{231}(t,x)=Q_{123}^{312}(t,x).
	\end{equation*}
\end{corollary}
%\newpage

Considering the distribution of patterns of length three, we only need to study the following $4$ generating functions for $\Sn{n}(132)$,
\begin{enumerate}[(1)]
	\item $Q_{132}^{123}(t,x)$,
	\item $Q_{132}^{213}(t,x)$,
	\item $Q_{132}^{231}(t,x)=Q_{132}^{312}(t,x)$,
	\item $Q_{132}^{321}(t,x)$,
\end{enumerate} 
and the following $3$ generating functions for $\Sn{n}(123)$,
\begin{enumerate}[(1)]
	\item $Q_{123}^{132}(t,x)=Q_{123}^{213}(t,x)$,
	\item $Q_{123}^{231}(t,x)=Q_{123}^{312}(t,x)$,
	\item $Q_{123}^{321}(t,x)$.
\end{enumerate}

It is easy to check that all the $7$ generating functions are different when looking at $\Sn{8}(132)$ and $\Sn{8}(123)$. Our motivation of this paper is to study the $7$ generating functions above, and then generalize some of the results. For example, we show that the function 
\begin{equation*}Q_n(x_1,\ldots,x_7)
:=Q_{n,132}^{\{12,21,123,213,231,312,321\}}(x_1,\ldots,x_7)
\end{equation*}
 satisfies the recursion that $Q_0(x_1,\ldots,x_7) =1$, and
\begin{multline*}
Q_{n}(x_1,\ldots,x_7)=\sum_{k=1}^{n} x_1^{k-1} x_2^{k(n-k)} x_5^{(k-1)(n-k)} 
\\\cdot Q_{k-1}(x_1 x_3 x_5^{(n-k)}, x_2 x_4 x_7^{(n-k)}, x_3,\ldots,x_7)
\cdot Q_{n-k}(x_1 x_6^{k}, x_2 x_7^{k}, x_3,\ldots,x_7).
\end{multline*}
The structure of this paper is as follows.
In Section 2, we introduce some background about permutations and two bijections between $\Sn{n}(132)$ and $\Sn{n}(123)$ and Dyck paths which are useful in our computation. Then we study length-three pattern distributions in $\Snn(132)$ in Section 3 and in $\Snn(123)$ in Section 4. In Section $5$, we show two applications of our results in computing pattern popularities.
Finally in Section $6$, we give a summary of this paper.

\section{Preliminaries}
Let $\sigma = \sg_1 \cdots \sg_n$ be a permutation written in one-line notation. 
The \emph{graph} of $\sg$, $G(\sg)$, is obtained by placing $\sg_i$ in the $i^{\mathrm{th}}$ column counting from left to right and $\sg_i^{\mathrm{th}}$ row counting from bottom to top on an $n\times n$ table for $i =1, \ldots, n$.
For example, the graph of the permutation 
$\sg = 471569283$ is pictured in \fref{graph}. 

We define $\inv(\sg): =\big|\{(i,j)|1\leq i<j\leq n,\ \sg_i>\sg_j\}\big|$ to be the number of \emph{inversions} and $\coinv(\sg): =\big|\{(i,j)|1\leq i<j\leq n,\ \sg_i<\sg_j\}\big|$ to be the number of \emph{coinversions} of a permutation $\sg$. Note that the number of inversions of a permutation is the same as the number of occurrences of the pattern $21$, and the number of coinversions of a permutation is the same as the number of occurrences of the pattern  $12$. Clearly, $\inv(\sg)+\coinv(\sg)=\binom{n}{2}$.
\begin{figure}[ht]
	\centering
	\vspace{-1mm}
	\begin{tikzpicture}[scale =.4]
	\draw[help lines] (0,0) grid (9,9);
	\filllll{1}{4};\filllll{2}{7};
	\filllll{3}{1};\filllll{4}{5};
	\filllll{5}{6};\filllll{6}{9};
	\filllll{7}{2};\filllll{8}{8};
	\filllll{9}{3};
	\end{tikzpicture}	
	\caption{The graph of $\sg=471569283$.}
	\label{fig:graph}
\end{figure}
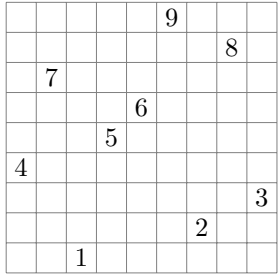

Given a permutation $\sg = \sg_1 \cdots \sg_n \in\Snn$, we say that $\sg_j$ is a \emph{left-to-right minimum} of $\sg$ if 
$\sg_i > \sg_j$ for all $i < j$. We let $\LRmin(\sg)$ denote the number of left-to-right minima of $\sg$. We shall also call each left-to-right minimum of $\sg$ a \emph{peak}, and the remaining numbers \emph{non-peaks} of $\sg$. We can see that the permutations in \fref{SnDn1} and \fref{SnDn2} both have peaks $\{8,6,4,3,2,1\}$.

Let $\pi=\pi_1\cdots\pi_m\in\Sn{m}$ and $\sg=\sg_1\cdots\sg_n\in\Snn$, then the \emph{direct sum} ($\pi\oplus\sg$) and the \emph{skew sum} ($\pi\ominus\sg$) of $\pi$ and $\sg$ are defined by 
\begin{eqnarray}
\pi\oplus\sg &: =& \pi_1\cdots\pi_m (\sg_1+m)\cdots(\sg_n+m),  \\
\pi\ominus\sg &: =&  (\pi_1+n)\cdots(\pi_m+n) \sg_1\cdots\sg_n.
\end{eqnarray}

Given an $n\times n$ square, we will label the coordinates of the columns from left to right
and the coordinates of the rows from top to bottom with $0,1, \ldots, n$ (different from the coordinates of the graph of a permutation).  An $(n,n)$-\emph{Dyck path} is a path made up of unit down-steps $D$ and unit right-steps $R$ which starts at $(0,0)$ and ends at $(n,n)$ and stays on or below the diagonal $y=x$ (these are ``down-right" Dyck paths). The set of $(n,n)$-Dyck paths is denoted by $\mathcal{D}_n$.

Given a Dyck path $P$, 
we let the first return of $P$, denoted by $\mathrm{ret}(P)$, be the smallest number $i>0$ such that $P$ goes through the point $(i,i)$.  For example, for 
$
P =DDRDDRRRDDRDRDRRDR
$
shown in  \fref{Dpath}, 
$\mathrm{ret}(P) =4$ since the leftmost point after $(0,0)$ on the diagonal that $P$ goes through is $(4,4)$. 
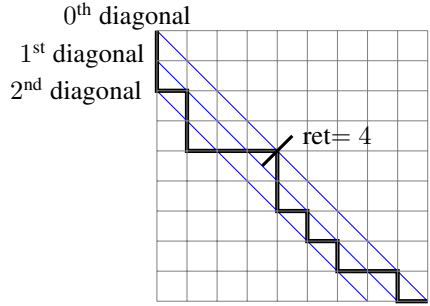
\begin{figure}[ht]
	\centering
	\vspace{-1mm}	
	\begin{tikzpicture}[scale =.4]
	\draw[blue] (0,9)--(9,0);
	\node[] at (-1,9.5) {$0^{\textnormal{th}}$ diagonal};
	\draw[blue] (0,8)--(8,0);
	\node[] at (-2.5,8.2) {$1^{\textnormal{st}}$ diagonal};
	\draw[blue] (0,7)--(7,0);
	\node[] at (-2.7,7) {$2^{\textnormal{nd}}$ diagonal};
	\path (-4,4.5);
	\draw[ultra thick] (0,9)--(0,7) -- (1,7)--(1,5)-- (4,5)--(4,3)--(5,3)--(5,2)--(6,2)--(6,1)--(8,1)--(8,0)--(9,0);
	\draw[help lines] (0,0) grid (9,9);	
	\draw[very thick] (3.5,4.5)--(4.5,5.5) node[right] {ret$=4$};
	\end{tikzpicture}
	\caption{A $(9,9)$-Dyck path $P=DDRDDRRRDDRDRDRRDR$.}
	\label{fig:Dpath}
	\vspace{-1mm}
\end{figure}

We refer to positions $(i,i)$ where $P$ goes through as \emph{return positions} of $P$. 
We call the full cells between $P$ and the main diagonal \emph{area cells}, and the cells below $P$ \emph{coarea cells}. Then we let $\area(P)$ and $\coarea(P)$ be the number of area cells and coarea cells of $P$. In the example in \fref{Dpath}, $\area(P)=7$ and $\coarea(P)=29$.

We shall also label the diagonals that go through corners of squares 
that are parallel to and below the main diagonal with $0, 1, 2, \ldots $ starting at the main diagonal, as shown in \fref{Dpath}. The \emph{peaks} of a path $P$ are the positions of consecutive $DR$ steps. 
We say that a peak is on the \thn{i} diagonal of $P$ if its $DR$ steps start and end at the \thn{i} diagonal. 
The path in \fref{Dpath} has four peaks on the first diagonal, one peak on the second diagonal and one peak on the main diagonal.
%In the path in \fref{Dpath}, the peaks are on the first, second, first, first, first and zeroth diagonal counting from top to bottom.

It is well known that for all $n \geq 1$, 
$|\Snn(132)|=|\Snn(123)|=|\mathcal{D}_n|=C_n$, where $C_n = \frac{1}{n+1}\binom{2n}{n}$ is the 
$n^{\mathrm{th}}$ Catalan number. Many bijections are known between these Catalan objects (see \cite{Stan}). We will apply the bijection of \cite{Kr} between $\Snn(132)$ and $\mathcal{D}_n$ and the bijection of \cite{EliDeu} between $\Snn(123)$ and $\mathcal{D}_n$ in our computation. The authors of this paper also discussed the two bijections in \cite{QR1,QR2} with more details.

%\subsubsection{The bijection $\Phi:\Snn(132) \rightarrow \mathcal{D}_n$}

We shall first describe 
the bijection $\Phi$ of \cite{Kr} between $\Snn(132)$ and $\mathcal{D}_n$. 
Given any permutation $\sg = \sg_1 \cdots \sg_n \in\Snn(132)$, we draw the graph $G(\sg)$ of $\sg$. Then, we 
shade the cells to the north-east of the cell that contains $\sg_i$. $\Phi(\sg)$ 
is the path that goes along the south-west boundary of the shaded cells. For example, this 
process is pictured in \fref{SnDn1} in the case where  $\sg=867943251\in\Sn{9}(132)$. 
In this case, $\Phi(\sg)=  DDRDDRRRDDRDRDRRDR$.  

\begin{figure}[ht]
	\centering
	\begin{tikzpicture}[scale =.35]
	\path[fill,black!15!white] (0,7) -- (1,7)--(1,5)-- (4,5)--(4,3)--(5,3)--(5,2)--(6,2)--(6,1)--(8,1)--(8,0)--(9,0)--(9,9)--(0,9);
	\draw[help lines] (0,0) grid (9,9);
	\filllll{1}{8};\filllll{2}{6};
	\filllll{3}{7};\filllll{4}{9};
	\filllll{5}{4};\filllll{6}{3};
	\filllll{7}{2};\filllll{8}{5};
	\filllll{9}{1};
	\path (0,-.5);
	\end{tikzpicture}	
	\begin{tikzpicture}[scale =.35]
	\draw[help lines] (0,0) grid (9,9);
	\fillgcross{1}{8};\fillgcross{2}{6};
	\fillcross{3}{7};\fillcross{4}{9};
	\fillgcross{5}{4};\fillgcross{6}{3};
	\fillgcross{7}{2};\fillcross{8}{5};
	\fillgcross{9}{1};	
	\draw (-1.5,4.5) node {$\rightarrow$};
	\path (-2.5,4.5);
	\path (0,-.5);
	\draw[very thick] (0,9)--(0,7) -- (1,7)--(1,5)-- (4,5)--(4,3)--(5,3)--(5,2)--(6,2)--(6,1)--(8,1)--(8,0)--(9,0);
	\path (10,0);
	\end{tikzpicture}
	\vspace{-1mm}
	\caption{The map $\Phi:\Snn(132)\rightarrow\mathcal{D}_n$.}
	\label{fig:SnDn1}
\end{figure}
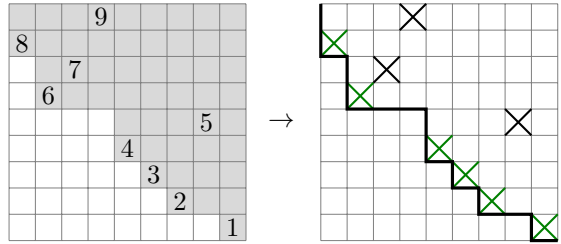

The \emph{horizontal segments} (or \emph{segments}) of the path $\Phi(\sg)$ are the maximal consecutive 
sequences of $R$ steps in $\Phi(\sg)$. For example, in \fref{SnDn1}, the lengths of 
the horizontal segments, reading from top to bottom, are $1,3,1,1,2,1$, and $\{6,7,9\}$ is the set of numbers associated with the second horizontal segment of $\Phi(\sg)$. 

The map $\Phi$ is invertible since for each Dyck path $P$, the peaks of $P$ give the left-to-right minima of the $132$-avoiding permutation, and the remaining numbers are uniquely determined by the left-to-right minima. More details about  $\Phi$ can be found in \cite{Kr}. We have the following properties for $\Phi$.

\ 
\begin{lem}\label{lem:p1}
	Let $P \in \mathcal{D}_n$ and $\sg = \Phi^{-1}(P)$. Then 
	\begin{enumerate}[(1)]
		\item for each horizontal segment $H$ of $P$, the set of numbers associated 
		to $H$ form a consecutive increasing sequence in $\sigma$ and the least 
		number of the sequence sits immediately above the first right-step of $H$. 
		%Hence the only decreases in $\sigma$ occur between two different horizontal segments of $P$. 
		
		\item The number $n$ is in the column of the last right-step before the first return. 
		
		\item Suppose that $\sigma_i$ is a peak of $\sigma$ and the cell containing $\sigma_i$ is on the 
		$k^{\mathrm{th}}$ diagonal. Then there are $k$ elements in the graph 
		$G(\sigma)$ in the first quadrant relative 
		to the coordinate system centered at $(i,\sigma_i)$. 
		
		\item  $\inv(\sg) = \coarea(P)$; $\coinv(\sg) = \area(P)$.
	\end{enumerate}
\end{lem}
\begin{proof}
	(1), (2) and (3) are proved in Lemma 3 in \cite{QR1}.
	
	For (4), it is clear that for any pair of indices  $i<j$, we have $\sg_i>\sg_j$ if and only if in path $P$, the \thn{i} column intersects the \thn{\sg_j} row at a coarea cell. Thus the number of inversions of $\sg$ is equal to the coarea of $P$, i.e.\ $\inv(\sg)=\coarea(P)$. Since $\inv(\sg)+\coinv(\sg)=\area(P)+\coarea(P)=\binom{n}{2}$, we have $\coinv(\sg)=\area(P)$.
\end{proof}

%\subsubsection{The bijection $\Psi:\Snn(123) \rightarrow \mathcal{D}_n$}
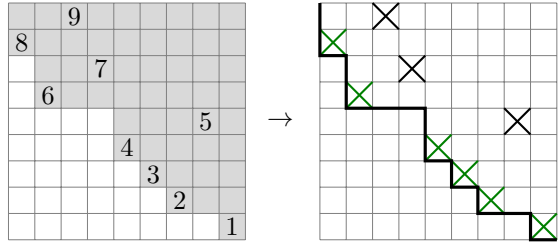
\begin{figure}[ht]
	\centering
	\begin{tikzpicture}[scale =.35]
	\path[fill,black!15!white] (0,7) -- (1,7)--(1,5)-- (4,5)--(4,3)--(5,3)--(5,2)--(6,2)--(6,1)--(8,1)--(8,0)--(9,0)--(9,9)--(0,9);
	\draw[help lines] (0,0) grid (9,9);
	\filllll{1}{8};\filllll{2}{6};
	\filllll{3}{9};\filllll{4}{7};
	\filllll{5}{4};\filllll{6}{3};
	\filllll{7}{2};\filllll{8}{5};
	\filllll{9}{1};
	\path (0,-.5);
	\end{tikzpicture}	
	\begin{tikzpicture}[scale =.35]
	\draw[help lines] (0,0) grid (9,9);
	\fillgcross{1}{8};\fillgcross{2}{6};
	\fillcross{3}{9};\fillcross{4}{7};
	\fillgcross{5}{4};\fillgcross{6}{3};
	\fillgcross{7}{2};\fillcross{8}{5};
	\fillgcross{9}{1};	
	\draw (-1.5,4.5) node {$\rightarrow$};
	\path (-2.5,4.5);
	\draw[very thick] (0,9)--(0,7) -- (1,7)--(1,5)-- (4,5)--(4,3)--(5,3)--(5,2)--(6,2)--(6,1)--(8,1)--(8,0)--(9,0);
	\path (0,-.5);
	\end{tikzpicture}
	\vspace{-1mm}
	\caption{The map $\Psi:\Snn(123)\rightarrow\mathcal{D}_n$.}
	\label{fig:SnDn2}
\end{figure}

The bijection $\Psi:\Snn(123) \rightarrow \mathcal{D}_n$  given 
by \cite{EliDeu} can be described in a similar way. 
Given any permutation $\sg \in\Snn(123)$, the Dyck path $\Psi(\sg)$ is constructed exactly as the bijection $\Phi$. \fref{SnDn2} shows an example of this map from $\sg=869743251\in\Sn{9}(123)$ to the Dyck path {\em DDRDDRRRDDRDRDRRDR}. The map $\Psi$ is invertible because each $123$-avoiding permutation has a unique left-to-right minima set. More details about  $\Psi$ can be found in \cite{EliDeu}. We then have the following lemma from \cite{QR1}.

\ 
\begin{lem}[\cite{QR1}, Lemma 4]\label{lem:p2} Let $P \in \mathcal{D}_n$ and $\sg = \Psi^{-1}(P)$. Then 
	\begin{enumerate}[(1)]
		\item for each horizontal segment $H$ of $P$, the least element of the 
		set of numbers associated to $H$ sits directly above the first right-step of $H$, and 
		the remaining numbers of the set form a consecutive decreasing sequence in $\sigma$. 
		
		\item $\sigma$ can be decomposed into two decreasing subsequences, the first decreasing 
		subsequence corresponds to the peaks of $\sigma$ and the second decreasing subsequence 
		corresponds to the non-peaks of $\sigma$. 
		
		\item Suppose that $\sigma_i$ is a peak of $\sigma$ and the cell containing $\sigma_i$ is on the 
		$k^{\mathrm{th}}$ diagonal. Then there are $k$ elements in the graph 
		$G(\sigma)$ in the first quadrant relative 
		to the coordinate system centered at $(i,\sigma_i)$.
	\end{enumerate}
\end{lem}

\section{The functions $Q_{132}^{\gamma}(t,x)$}

In this section, we give the recursive relations for generating functions $Q_{132}^{\gamma}(t,x)$ by examining the set $\Snn(132)$. We shall first look at the structure of a $132$-avoiding permutation.

Given $\sg=\sg_1\cdots\sg_n\in\Snn(132)$, we suppose $\sg_k=n$ is the biggest number in the permutation. The numbers $\sg_1,\ldots,\sg_{k-1}$ must be bigger than the numbers $\sg_{k+1},\ldots,\sg_n$ since otherwise there will be a  $132$ pattern in $\sg$. Thus we can break the permutation $\sg$ into three parts: the first $k-1$ numbers, the biggest number $\sg_k=n$, and the last $n-k$ numbers. 
We let $A(\sg) =\red(\sg_1\cdots\sg_{k-1})$ and $B(\sg) =\red(\sg_{k+1}\cdots\sg_n)$ be the reduction of the first $k-1$ numbers and the last $n-k$ numbers, then $A(\sg)\in\Sn{k-1}(132)$ and $B(\sg)\in\Sn{n-k}(132)$. The left picture of \fref{SnDn1} is an example for $\sg=867943251\in\Sn{9}(132)$ with $A(\sg)=312$ and $B(\sg)=43251$. We also let $\Abar(\sg): =\red(\sg_1\cdots\sg_{k})$ be the reduction of the first $k$ numbers. The structure of $\sg$ is shown in \fref{Sn132structure}.

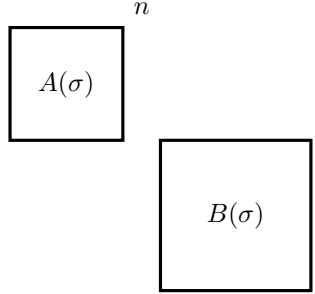
\begin{figure}[ht]
	\centering
	\begin{tikzpicture}[scale =.5]
	%		\draw[help lines] (0,0) grid (8,8);
	\draw[very thick] (0,4) rectangle (3,7);
	\draw[very thick] (4,0) rectangle (8,4);
	\node at (1.5,5.5) {$A(\sigma)$};
	\node at (6,2) {$B(\sigma)$};
	%		\fillll{1}{0}{1};
	%		\fillll{0}{1}{1};
	%		\fillll{8}{0}{n};
	\fillll{4}{8}{n};
	\end{tikzpicture}
	\vspace{-3mm}
	\caption{The structure of $\sg\in\Snn(132)$.}
	\label{fig:Sn132structure}
\end{figure}

Now we count the number of occurrences of a pattern $\ga=\ga_1\cdots\ga_r\in\Sn{r}(132)$ in $\sg\in\Snn(132)$. The case when $r=1$ is trivial and $\occr_1(\sg)=n$. If $r>1$, we shall count $\ga$ from the three parts, $A(\sg)$, $\sg_k$ and $B(\sg)$. First, there are $\left(\occr_{\ga}(A(\sg))+\occr_{\ga}(B(\sg))\right)$ occurrences of $\ga$ in parts $A(\sg)$ and $B(\sg)$. Then we count occurrences of $\ga$ that intersect with at least two of the three parts of $\sg$.

Similar to $\sg$, we shall break $\ga$ into three parts: $A(\ga)=\red(\ga_1\cdots\ga_{s-1})$, $\ga_s=r$ and $B(\ga) =\red(\ga_{s+1}\cdots\ga_r)$. We also let $\Abar(\ga)=\red(\ga_1\cdots\ga_{s})$. Let $\chi(x)$ be the function that takes value $1$ when the statement $x$ is true and $0$ otherwise. Then there are 
\begin{enumerate}[(a)]
	\item $\chi(s=r) \cdot \occr_{A(\ga)}(A(\sg))$ occurrences of $\ga$ stretching over parts $A(\sg)$ and $\sg_k$, 
	\item $\chi(s=1) \cdot \occr_{B(\ga)}(B(\sg))$ occurrences of $\ga$ stretching over parts  $\sg_k$ and $B(\sg)$, 
	\item $\chi(s<r) \cdot \occr_{\Abar(\ga)}(A(\sg)) \cdot \occr_{B(\ga)}(B(\sg))$ occurrences of $\ga$ stretching over parts  $A(\sg)$ and $B(\sg)$ if $\ga_r=r-s$, 
	\item and $\chi(1<s<r) \cdot \occr_{A(\ga)}(A(\sg)) \cdot \occr_{B(\ga)}(B(\sg))$ occurrences of $\ga$ stretching over all three parts.
\end{enumerate}
Note that (c) requires $\ga_r=r-s$, i.e.\ the permutation $B(\ga)$ cannot be expressed as the skew sum of two smaller permutations. If $\ga_r\neq r-s$, then
\begin{enumerate}[(a')]
	\addtocounter{enumi}{2}
	\item if $\{ \pi_1\ominus\tau_1,\ldots, \pi_j \ominus\tau_j \}$ is the collection of ways to write $\ga$ as the skew sum of two smaller permutations, then the number of occurrences of $\ga$ stretching over parts  $A(\sg)$ and $B(\sg)$ is
	\begin{equation*}
	\sum_{i=1}^{j} \occr_{\pi_i}(A(\sg)) \cdot \occr_{\tau_i}(B(\sg)).
	\end{equation*}
\end{enumerate}
We name the method above the \emph{recursive counting method}, which allows us to count $\occr_\ga(\sg)$ by counting pattern occurrences from the components of $\sg$. 

\subsection{The function $Q_{n,132}^{\{12,21\}}(x_1,x_2)$}
As a first application of our recursive counting method, we have the following theorem first proved by F\"{u}rlinger and Hofbauer \cite{FH} in 1985 about the distribution of patterns of length two.

\ 
\begin{theorem}[F\"{u}rlinger and Hofbauer]\label{theorem:Q12}
	Let $Q_n(x_1,x_2): =Q_{n,132}^{\{12,21\}}(x_1,x_2)$ and \\ $Q(t,x_1,x_2): =Q_{132}^{\{12,21\}}(t,x_1,x_2)$, then
	\begin{equation}\label{Q12}
	Q_0(x_1,x_2)=1, \ Q_n(x_1,x_2)=\sum_{k=1}^{n}x_1^{k-1}x_2^{k(n-k)}Q_{k-1}(x_1,x_2)Q_{n-k}(x_1,x_2),
	\end{equation} 
	and
	\begin{equation}\label{Q12function}
	Q(t,x,1)=1+t Q(t,x,1)\cdot Q(tx,x,1).
	\end{equation}
\end{theorem}
The theorem was initially proved by enumerating area and coarea statistics of Dyck paths. 
Here we shall give a brief proof using permutations.

\begin{proof}
	To prove equation (\ref{Q12}), we shall consider the distribution of the patterns $\ga=12$ and $\tau=21$ in $\Snn(132)$ using the recursive counting method.
	
	Given $\sg\in\Snn(132)$ such that $\sg_k=n$, we break $\sg$ into 3 parts: $A(\sg),\ \sg_k$ and $B(\sg)$. We have $A(\sg)\in\Sn{k-1}(132)$ and $B(\sg)\in\Sn{n-k}(132)$. By (a) of the recursive counting method, 
	\begin{eqnarray}\label{12}
	\occr_{12}(\sg)&=&\occr_{12}(A(\sg))+\occr_{12}(B(\sg))+\occr_1(A(\sg))\nonumber\\
	&=&\occr_{12}(A(\sg))+\occr_{12}(B(\sg))+k-1,
	\end{eqnarray}
	and
	\begin{eqnarray}\label{21}
	\occr_{21}(\sg)&=&\occr_{21}(A(\sg))+\occr_{21}(B(\sg))+\occr_{1}(B(\sg))+\occr_{1}(A(\sg)) \cdot \occr_{1}(B(\sg))\nonumber\\
	&=&\occr_{21}(A(\sg))+\occr_{21}(B(\sg))+k(n-k).
	\end{eqnarray}

	Thus,
	\begin{eqnarray*}
		&&Q_n(x_1,x_2) =  \sum_{\sg\in\Sn{n}(132)} x_1^{\occr_{12}(\sg)} x_2^{\occr_{21}(\sg)}  \nonumber\\
		&=& \sum_{k=1}^{n} \sum_{\sg\in\Sn{n}(132),\sg_k=n} x_1^{\occr_{12}(A(\sg))+\occr_{12}(B(\sg))+k-1} x_2^{\occr_{21}(A(\sg))+\occr_{21}(B(\sg))+k(n-k)} \nonumber\\
		&=& \sum_{k=1}^{n} x_1^{k-1}x_2^{k(n-k)} \sum_{\sg\in\Sn{n}(132),\sg_k=n} x_1^{\occr_{12}(A(\sg))} x_2^{\occr_{21}(A(\sg))} x_1^{\occr_{12}(B(\sg))} x_2^{\occr_{21}(B(\sg))}  \nonumber\\
		&=& \sum_{k=1}^{n} x_1^{k-1}x_2^{k(n-k)} \sum_{\pi\in\Sn{k-1}(132)} x_1^{\occr_{12}(\pi)} x_2^{\occr_{21}(\pi)} \sum_{\tau\in\Sn{n-k}(132)} x_1^{\occr_{12}(\tau)} x_2^{\occr_{21}(\tau)}  \nonumber\\
		&=& \sum_{k=1}^{n} x_1^{k-1}x_2^{k(n-k)} Q_{k-1}(x_1,x_2)Q_{n-k}(x_1,x_2),
	\end{eqnarray*}
	which proves equation (\ref{Q12}). Equation (\ref{Q12function}) is a consequence of equation (\ref{Q12}). More explicitly, by equation (\ref{Q12}),
	\begin{equation}
	Q_n(x,1)=\sum_{k=1}^{n}x^{k-1}Q_{k-1}(x,1)Q_{n-k}(x,1)
	\end{equation}
	for any $n>0$. Thus,
	\begin{eqnarray*}
		&&Q(t,x,1)=\sum_{n\geq 0}t^n Q_n(x,1)\\
		&=&\sum_{n\geq 0}t^n \sum_{k=1}^{n}x^{k-1}Q_{k-1}(x,1)Q_{n-k}(x,1)\\
		&=&1+ t \sum_{n\geq 0} \sum_{k=1}^{n} t^{{k-1}}x^{k-1}Q_{k-1}(x,1) t^{n-k} Q_{n-k}(x,1)\\
		&=&1+ t \left(\sum_{n\geq 0}(tx)^n Q_n(x,1)\right) \left(\sum_{n\geq 0}t^n Q_n(x,1)\right)\\
		&=&1+t Q(tx,x,1)\cdot Q(t,x,1).\qedhere
	\end{eqnarray*}
\end{proof}

Using the recursive equation (\ref{Q12}), we can use Mathematica to compute the Taylor series of the function $Q_{132}^{\{12,21\}}(t,x_1,x_2)$ as follows.
\begin{multline}\label{1221}
Q_{132}^{\{12,21\}}(t,x_1,x_2)=1 + t + t^2 ( x_1 +x_2) + t^3 ( x_1^3+ x_1^2 x_2+2 x_1 x_2^2 +x_2^3) + 
t^4 (x_1^6+x_1^5 x_2\\+2 x_1^4x_2^2+3 x_1^3 x_2^3+ 3 x_1^2 x_2^4+ 3 x_1 x_2^5+x_2^6) + 
t^5 ( x_1^{10} +x_1^9 x_2+2 x_1^8 x_2^2\\+ 3 x_1^7 x_2^3+5 x_1^6 x_2^4+5 x_1^5 x_2^5+7 x_1^4 x_2^6+ 7 x_1^3 x_2^7+6 x_1^2 x_2^8+4 x_1x_2^9+x_2^{10}) + \cdots.
\end{multline}
Setting $x_2=1$ and $x_1=x$ in equation (\ref{1221}), we have
\begin{multline}
Q_{132}^{12}(t,x)=1 + t + t^2 (1 + x) + t^3 (1 + 2 x + x^2 + x^3) + 
t^4 (1 + 3 x + 3 x^2 + 3 x^3 + 2 x^4 + x^5 + x^6) \\+ 
t^5 (1 + 4 x + 6 x^2 + 7 x^3 + 7 x^4 + 5 x^5 + 5 x^6 + 3 x^7 + 
2 x^8 + x^9 + x^{10}) + \cdots.
\end{multline}

\subsection{The function $Q_{n,132}^{\{12,21,123,213,231,312,321\}}(x_1,\ldots,x_7)$}
Let $\Ga_2=\{12,21\}$ and $\Ga_3=\{123,213,231,312,321\}$ be sets of permutation patterns, then the function
\begin{equation*}
Q_{n,132}^{\Ga_2\cup\Ga_3}(x_1,\ldots,x_7)=Q_{n,132}^{\{12,21,123,213,231,312,321\}}(x_1,\ldots,x_7)
\end{equation*}
 tracks all patterns of length two or three of permutations in $\Snn(132)$. 
We shall prove the following theorem. We use the shorthand $Q_n(x_1,\ldots,x_7)$ for $Q_{n,132}^{\Ga_2\cup\Ga_3}(x_1,\ldots,x_7)$. 

\ 
\begin{theorem}\label{theorem:S3Sn132}The function $Q_{n,132}^{\Ga_2\cup\Ga_3}(x_1,\ldots,x_7)$ satisfies the recursion
	
	\hspace*{3mm}$\displaystyle Q_0(x_1,\ldots,x_7) =1,$
	\begin{multline}\label{eqnS3}
	Q_{n}(x_1,\ldots,x_7)=\sum_{k=1}^{n} x_1^{k-1} x_2^{k(n-k)} x_5^{(k-1)(n-k)} 
	\\\cdot Q_{k-1}(x_1 x_3 x_5^{(n-k)}, x_2 x_4 x_7^{(n-k)}, x_3,\ldots,x_7)
	\cdot Q_{n-k}(x_1 x_6^{k}, x_2 x_7^{k}, x_3,\ldots,x_7).
	\end{multline}
\end{theorem}
\begin{proof}
	We shall consider the distribution of the patterns $\ga_1=12$, $\ga_2= 21$, $\ga_3= 123$, $\ga_4= 213$, $\ga_5= 231$, $\ga_6= 312$ and $\ga_7= 321$ in $\Snn(132)$.
	
	Given $\sg\in\Snn(132)$ such that $\sg_k=n$. Similar to \tref{Q12}, we make permutations $A(\sg)\in\Sn{k-1}(132)$ and $B(\sg)\in\Sn{n-k}(132)$.
	The number of occurrences of $\ga_1$ and $\ga_2$ is given by equation (\ref{12}) and (\ref{21}).
	For the patterns $\ga_3,\ \ga_4,\ \ga_5,\ \ga_6$ and $\ga_7$, we have the following formulas from the recursive counting method.
	\begin{eqnarray}\label{o123}
	\occr_{123}(\sg)&=&\occr_{123}(A(\sg))+\occr_{123}(B(\sg))+\occr_{12}(A(\sg)),
	\\\label{o213}
	\occr_{213}(\sg)&=&\occr_{213}(A(\sg))+\occr_{213}(B(\sg))+\occr_{21}(A(\sg)),
	\\\label{o231}
	\occr_{231}(\sg)&=&\occr_{231}(A(\sg))+\occr_{231}(B(\sg))
	\nonumber\\&&+\occr_{12}(A(\sg))\cdot\occr_1(B(\sg))+\occr_{1}(A(\sg))\cdot\occr_1(B(\sg))\nonumber\\
	&=&\occr_{231}(A(\sg))+\occr_{231}(B(\sg))+(n-k)\occr_{12}(A(\sg))+(k-1)(n-k),
	\\\label{o312}
	\occr_{312}(\sg)&=&\occr_{312}(A(\sg))+\occr_{312}(B(\sg))+\occr_{12}(B(\sg))+\occr_{1}(A(\sg))\cdot\occr_{12}(B(\sg))\nonumber\\
	&=&\occr_{312}(A(\sg))+\occr_{312}(B(\sg))+k\cdot\occr_{12}(B(\sg)),\mbox{ \ \ \  and}
	\\\label{o321}
	\occr_{321}(\sg)&=&\occr_{321}(A(\sg))+\occr_{321}(B(\sg))+\occr_{21}(B(\sg))
	\nonumber\\&&+\occr_{1}(A(\sg))\cdot\occr_{21}(B(\sg))+\occr_{21}(A(\sg))\cdot\occr_{1}(B(\sg))\nonumber\\
	&=&\occr_{321}(A(\sg))+\occr_{321}(B(\sg))+k\cdot\occr_{21}(B(\sg))+(n-k)\occr_{21}(A(\sg)).
	\end{eqnarray}
	Thus,
	\begin{eqnarray*}\label{S3arrange}
		&&Q_n(x_1,\ldots,x_7) \nonumber\\
		&=&  \sum_{\sg\in\Sn{n}(132)} x_1^{\occr_{12}(\sg)} x_2^{\occr_{21}(\sg)} x_3^{\occr_{123}(\sg)} x_4^{\occr_{213}(\sg)} x_5^{\occr_{231}(\sg)} x_6^{\occr_{312}(\sg)} x_7^{\occr_{321}(\sg)}  \nonumber\\
		&=& \sum_{k=1}^{n} \sum_{\pi\in\Sn{k-1}(132)} \sum_{\tau\in\Sn{n-k}(132)} 
		x_1^{\occr_{12}(\pi)+\occr_{12}(\tau)+k-1}
		x_2^{\occr_{21}(\pi)+\occr_{21}(\tau)+k(n-k)}\nonumber\\
		&&\cdot
		x_3^{\occr_{123}(\pi)+\occr_{123}(\tau)+\occr_{12}(\pi)}
		x_4^{\occr_{213}(\pi)+\occr_{213}(\tau)+\occr_{21}(\pi)}\nonumber\\
		&&\cdot
		x_5^{\occr_{231}(\pi)+\occr_{231}(\tau)+(n-k)\occr_{12}(\pi)+(k-1)(n-k)}
		x_6^{\occr_{312}(\pi)+\occr_{312}(\tau)+k\cdot\occr_{12}(\tau),}\nonumber\\
		&&\cdot
		x_7^{\occr_{321}(\pi)+\occr_{321}(\tau)+k\cdot\occr_{21}(\tau)+(n-k)\occr_{21}(\pi)}
		\nonumber\\
		&=&\sum_{k=1}^{n} x_1^{k-1} x_2^{k(n-k)} x_5^{(k-1)(n-k)} 
		Q_{k-1}(x_1 x_3 x_5^{(n-k)}, x_2 x_4 x_7^{(n-k)}, x_3,\ldots,x_7)
		\nonumber\\
		&&\cdot Q_{n-k}(x_1 x_6^{k}, x_2 x_7^{k}, x_3,\ldots,x_7).
		\qedhere
	\end{eqnarray*}
\end{proof}

Using Mathematica, one can efficiently compute the polynomial\\ $Q_{132}^{\Ga_2\cup\Ga_3}(t,x_1,\ldots,x_7)=\sum_{n\geq 0}t^n Q_{n,132}^{\Ga_2\cup\Ga_3}(x_1,\ldots,x_7)$ as follows. 
\begin{multline}
Q_{132}^{\Ga_2\cup\Ga_3}(t,x_1,\ldots,x_7) = 1+t+t^2(x_1+x_2) + t^3(x_1^3 x_3+x_1^2 x_2 x_4+x_1 x_2^2 x_5+x_1 x_2^2 x_6+x_2^3 x_7)\\
+t^4\left(x_1^6 x_3^4+x_1^5 x_2 x_3^2 x_4^2+x_1^4 x_2^2 x_3 x_4^2 x_5+x_1^3 x_2^3 x_3 x_5^3+x_1^4 x_2^2 x_3 x_4^2 x_6+x_1^2 x_2^4 x_5^2 x_6^2+x_1^3 x_2^3 x_3 x_6^3\right.
\\\left.+x_1^3 x_2^3 x_4^3 x_7+x_1^2 x_2^4 x_4 x_5^2 x_7+x_1^2 x_2^4 x_4 x_6^2 x_7+x_1 x_2^5 x_5^2 x_7^2+x_1 x_2^5 x_5 x_6 x_7^2+x_1 x_2^5 x_6^2 x_7^2+x_2^6 x_7^4\right)\\
+t^5\left(x_1^{10} x_3^{10}+x_1^9 x_2 x_3^7 x_4^3+x_1^8 x_2^2 x_3^5 x_4^4 x_5+x_1^7 x_2^3 x_3^4 x_4^3 x_5^3+x_1^6 x_2^4 x_3^4 x_5^6+x_1^8 x_2^2 x_3^5 x_4^4 x_6+x_1^6 x_2^4 x_3^2 x_4^4 x_5^2 x_6^2\right.
\\+x_1^7 x_2^3 x_3^4 x_4^3 x_6^3+x_1^4 x_2^6 x_3 x_5^6 x_6^3+x_1^6 x_2^4 x_3^4 x_6^6+x_1^4 x_2^6 x_3 x_5^3 x_6^6+x_1^7 x_2^3 x_3^3 x_4^6 x_7+x_1^6 x_2^4 x_3^2 x_4^5 x_5^2 x_7
\\+x_1^5 x_2^5 x_3^2 x_4^2 x_5^5 x_7+x_1^6 x_2^4 x_3^2 x_4^5 x_6^2 x_7+x_1^5 x_2^5 x_3^2 x_4^2 x_6^5 x_7+x_1^5 x_2^5 x_3 x_4^5 x_5^2 x_7^2+x_1^4 x_2^6 x_3 x_4^2 x_5^5 x_7^2
\\+x_1^5 x_2^5 x_3 x_4^5 x_5 x_6 x_7^2+x_1^4 x_2^6 x_3 x_4^2 x_5^4 x_6 x_7^2+x_1^5 x_2^5 x_3 x_4^5 x_6^2 x_7^2+x_1^3 x_2^7 x_4 x_5^4 x_6^3 x_7^2+x_1^4 x_2^6 x_3 x_4^2 x_5 x_6^4 x_7^2
\\+x_1^3 x_2^7 x_4 x_5^3 x_6^4 x_7^2+x_1^4 x_2^6 x_3 x_4^2 x_6^5 x_7^2+x_1^3 x_2^7 x_3 x_5^6 x_7^3+x_1^3 x_2^7 x_3 x_5^3 x_6^3 x_7^3+x_1^3 x_2^7 x_3 x_6^6 x_7^3+x_1^4 x_2^6 x_4^6 x_7^4
\\+x_1^3 x_2^7 x_4^3 x_5^3 x_7^4+x_1^2 x_2^8 x_5^4 x_6^2 x_7^4+x_1^3 x_2^7 x_4^3 x_6^3 x_7^4+x_1^2 x_2^8 x_5^3 x_6^3 x_7^4+x_1^2 x_2^8 x_5^2 x_6^4 x_7^4+x_1^2 x_2^8 x_4 x_5^4 x_7^5
\\\left.+x_1^2 x_2^8 x_4 x_5^2 x_6^2 x_7^5+x_1^2 x_2^8 x_4 x_6^4 x_7^5+x_1 x_2^9 x_5^3 x_7^7+x_1 x_2^9 x_5^2 x_6 x_7^7+x_1 x_2^9 x_5 x_6^2 x_7^7+x_1 x_2^9 x_6^3 x_7^7+x_2^{10} x_7^{10}\right)\\+\cdots.
\end{multline}

Let 
\begin{equation*}
P_{n}^{\gamma}(q,x) : = \sum_{\sg\in\Sn{n}(132)} q^{\coinv(\sg)}x^{\occr_{\gamma}(\sg)},
\end{equation*}

We can evaluate appropriate variables of $Q_{n,132}^{\Ga_2\cup\Ga_3}(x_1,\ldots,x_7)$ at $1$ and use the relation that $\inv(\sg)+\coinv(\sg)=\binom{n}{2}$ to get the following corollary. We use the shorthand $P_n$ for $P_n^\ga$ in the RHS of each equation. 

\ 
\begin{corollary}\label{corollary:S3}
	We have the following equations.
	\begin{eqnarray}
	\label{S31}	P_{0}^{\gamma}(q,x)&=&1\textnormal{ \ \ \ for each pattern }\gamma,\\
	\label{S32}	P_{n}^{123}(q,x) &=& \sum_{k=1}^{n} q^{k-1} P_{k-1}(qx,x) P_{n-k}(q,x),\\
	\label{S33}	P_{n}^{213}(q,x) &=& \sum_{k=1}^{n} q^{k-1}x^{\frac{(k-1)(k-2)}{2}} P_{k-1}\left(\frac{q}{x},x\right) P_{n-k}(q,x),\\
	\label{S34}	P_{n}^{231}(q,x) &=& \sum_{k=1}^{n} q^{k-1}x^{(k-1)(n-k)} P_{k-1}\left(qx^{(n-k)},x\right) P_{n-k}(q,x),\\
	\label{S366}	P_{n}^{312}(q,x) &=& \sum_{k=1}^{n} q^{k-1} P_{k-1}\left(q,x\right) P_{n-k}(q x^k,x),\qquad \mbox{ \ \ and}\\
	\label{S35}	P_{n}^{321}(q,x) &=& \sum_{k=1}^{n} q^{k-1}x^{\frac{(n-k)(kn-4k+2)}{2}} P_{k-1}\left(\frac{q}{x^{n-k}},x\right) P_{n-k}\left(\frac{q}{x^{k}},x\right).
	\end{eqnarray}
\end{corollary}

Note that as a consequence of \lref{1}, the two polynomials $P_{n}^{231}(q,x)$ and $P_{n}^{312}(q,x)$ are identical, and \coref{S3} gives two different recursions for such polynomials.

We can also compute $Q_{132}^{\ga}(t,x)=\sum_{n\geq 0} t^n P_{n}^{\gamma}(1,x)$ where $\ga$ has length three as follows.
\begin{multline}
Q_{132}^{123}(t,x)=1 + t + 2 t^2 + t^3 (4 + x) + t^4 (8 + 4 x + x^2 + x^4) + 
t^5 (16 + 12 x + 5 x^2 + x^3 + 4 x^4 + 2 x^5 + x^7 + x^{10}) \\+ 
t^6 (32 + 32 x + 18 x^2 + 6 x^3 + 13 x^4 + 10 x^5 + 3 x^6 + 4 x^7 + 
3 x^8 + 5 x^{10} + 2 x^{11} + 2 x^{13} + x^{16} + x^{20}) +\cdots,
\end{multline}
%\vspace*{-6mm}
\begin{multline}
Q_{132}^{213}(t,x)=1 + t + 2 t^2 + t^3 (4 + x) + t^4 (8 + 2 x + 3 x^2 + x^3) + 
t^5 (16 + 5 x + 6 x^2 + 5 x^3 + 3 x^4 + 5 x^5 + 2 x^6)
\\ + 
t^6 (32 + 12 x + 16 x^2 + 11 x^3 + 9 x^4 + 10 x^5 + 10 x^6 + 5 x^7 + 
10 x^8 + 10 x^9 + 6 x^{10} + x^{12})+\cdots,
\end{multline}
%\vspace*{-7mm}
\begin{multline}
Q_{132}^{231}(t,x)=Q_{132}^{312}(t,x)\\=1 + t + 2 t^2 + t^3 (4 + x) + t^4 (8 + 2 x + 3 x^2 + x^3) + 
t^5 (16 + 4 x + 6 x^2 + 7 x^3 + 4 x^4 + 2 x^5 + 3 x^6) \\+ 
t^6 (32 + 8 x + 12 x^2 + 14 x^3 + 17 x^4 + 7 x^5 + 17 x^6 + 5 x^7 + 
5 x^8 + 8 x^9 + 5 x^10 + 2 x^12) +\cdots,
\end{multline}
%\vspace*{-5mm}
\begin{multline}
Q_{132}^{321}(t,x)= 1 + t + 2 t^2 + t^3 (4 + x) + t^4 (7 + 3 x + 3 x^2 + x^4) + 
t^5 (11 + 5 x + 9 x^2 + 3 x^3 + 6 x^4 + 3 x^5 + 4 x^7 + x^{10}) \\+ 
t^6 (16 + 7 x + 15 x^2 + 9 x^3 + 17 x^4 + 7 x^5 + 10 x^6 + 12 x^7 + 
7 x^8 + 6 x^9 + 7 x^{10} + 3 x^{11} + 6 x^{12} + 4 x^{13} + 5 x^{16} + x^{20})\\ +\cdots.
\end{multline}

\subsection{Longer patterns whose distributions satisfy good recursions}

We have  built recursions for generating functions which give the distribution of all patterns of length two or three in $\Snn(132)$. This leads to a natural question --- can we give recursions for the generating functions tracking any pattern in $\Snn(132)$ like we have done in Section 3.1 and Section 3.2? We notice that though we can always use the recursive counting method, we do not always obtain clear recursions like \tref{Q12} and \tref{S3Sn132}.

Let $\ga$ be a permutation pattern. We say that the distribution of the pattern $\ga$ in $\Snn(132)$ \emph{satisfies a good recursion} if there exist $s$ permutations $\ga_1,\ldots,\ga_s$ of length at least two such that the generating function $Q_n(x,x_1,\ldots,x_s)=Q_{n,132}^{\ga,\ga_1,\ldots,\ga_s}(x,x_1,\ldots,x_s)$
satisfies that $Q_0(x,x_1,\ldots,x_s)=1$, and
\begin{equation}\label{good}
Q_n(x,x_1,\ldots,x_s)=\sum_{i=1}^{n}q(X)Q_{i-1}(p_1(X),\ldots,p_{s+1}(X)) Q_{n-i}(q_1(X),\ldots,q_{s+1}(X))
\end{equation}
for $n\geq 1$,
where $X=\{x,x_1,\ldots,x_s\}$, and $q(X),p_1(X),\ldots,p_{s+1}(X),q_1(X),\ldots,q_{s+1}(X)$ are $2s+3$ rational functions about variables in $X$ such that the powers of variables in the numerator and denominator are polynomials of $n$ and $i$.

Thus, for any pattern $\ga$ whose distribution satisfies a good recursion, there is a $Q_n(x,x_1,\ldots,x_s)$ defined as above such that it can be computed recursively from the functions $Q_i(x,x_1,\ldots,x_s)$ for $i=0,\ldots,n-1$. We have the following theorem about the number of permutations in $\Snn(132)$ whose distributions satisfy good recursions.

\ 
\begin{theorem}\label{theorem:an}
	Let $\{a_n\}_{n\geq0}$ be the integer sequence defined by 
	\begin{equation}
	a_0=a_1=1,\ a_2=2,\mbox{ and }a_n=a_{n-1}+2a_{n-2}+a_{n-3}.
	\end{equation}
	Then the number of permutations in $\Snn(132)$ whose distributions satisfy good recursions is at least $a_n$.
\end{theorem}
\begin{proof}
	Given $\sg=\sg_1\cdots\sg_n\in\Snn(132)$, we define
	\begin{eqnarray}
	\sg'&: =&\sg_1\cdots\sg_n (n+1),\\
	\sg''&: =&(n+2)\sg_1\cdots\sg_n (n+1),\\
	\sg'''&: =&(\sg_1+1)\cdots(\sg_n+1) (n+2)1,\mbox{ \ \ and }\\
	\sg''''&: =&(n+3)(\sg_1+1)\cdots(\sg_n+1) (n+2)1.
	\end{eqnarray}
	
	We construct a set $\Ga$ of permutation patterns as follows. We first let the empty permutation $\emptyset$ be an element of $\Ga$. Next, for each permutation $\sg\in\Ga$, we let $\sg',\sg'',\sg''',\sg''''\in\Ga$. Clearly, each permutation in $\Ga$ is 132-avoiding, and the number of permutations in $\Ga\cap\Snn$ is $a_n$ based on the recursive construction of the set $\Ga$. 
	
	From the recursive counting method, the distributions of $\sg',\sg'',\sg''',\sg''''$ satisfy good recursions as long as the distribution of $\sg$ satisfies a good recursion. Thus
	the distribution of each permutation in $\Ga$ satisfies a good recursion, which proves the theorem.
	
	In fact, when we count the number of occurrences of $\ga\in\Sn{r}(132)$ in $\sg\in\Snn{(132)}$ with recursive counting method, we shall break the pattern $\ga$ into three parts:  $A(\ga)$, $r$ and $B(\ga)$. 
	$\ga$ fails to satisfy a good recursion by the case (c') of the recursive counting method if $\ga=\pi\ominus\tau$ for some permutations $\pi,\tau$ of length at least two (in this case the generating function cannot be recursively computed in the way of equation (\ref{good}) since the RHS of equation (\ref{good}) never gives the product $\occr_{\pi}(A(\sg)) \cdot \occr_{\tau}(B(\sg))$). This is saying that
	we cannot have $|A(\ga)|\geq 1$ and $|B(\ga)|\geq 2$ simultaneously. 
	
	If $B(\ga)$ is empty, then $\ga=A(\ga)'$. If $B(\ga)=1$, then $\ga=A(\ga)'''$.
	If $|B(\ga)|\geq2$ and $A(\ga)$ is empty, then we shall decompose $B(\ga)$ into three parts: $A(B(\ga))$, $r-1$ and $B(B(\ga))$. If $B(B(\ga))$ is empty, then $\ga=A(B(\ga))''$. If $B(B(\ga))$ is not empty, then $B(B(\ga))$ can only be of size one to make $\ga$ not separable into a skew sum of two nontrivial permutations, and $\ga=A(B(\ga))''''$. 
	
	Thus, $\Ga$ collects all permutations that satisfy good recursions
	if we use the recursive counting method, and there are exactly $a_n$ permutations in $\Snn$ whose distributions satisfy good recursions using the recursive counting method (there might be more permutations in $\Snn$ whose distributions satisfy good recursions, but for other reasons not accessible by the recursive counting method).
\end{proof}

The sequence $\{a_n\}_{n\geq0}=\{1,1,2,5,10,22,47,101,217,\ldots\}$ appears in OEIS of \cite{OEIS} as sequence A101399.

We shall give an example of a longer pattern $\ga=12\cdots m$ whose distribution satisfies a good recursion in the following theorem. Note that this gives a way to count the number of occurrences of $12\cdots m$ in $\Snn(132)$ different from that of \cite{MV1}.

\ 
\begin{theorem}\label{theorem:Sn132m}
	Given $m\geq 2$ and $n\geq 0$, we let
	\begin{eqnarray}
	Q^{(m)}_{n,132}(x_2,x_3,\ldots,x_m)&: =&\sum_{\sg\in\Snn(132)}x_2^{\occr_{12}(\sg)}x_3^{\occr_{123}(\sg)}\cdots x_m^{\occr_{12\cdots m}(\sg)}\mbox{ and}\\
	Q^{(m)}_{132}(t,x_2,x_3,\ldots,x_m)&: =&\sum_{n\geq 0}t^nQ^{(m)}_{n,132}(x_2,x_3,\ldots,x_m),
	\end{eqnarray}
	then we have the following equations,
	\begin{eqnarray}
	\label{132m1} Q^{(m)}_{n,132}(x_2,\ldots,x_m)&\hspace*{-3mm}=&\hspace*{-3mm}\sum_{k=1}^{n} x_2^{k-1}Q^{(m)}_{k-1,132}(x_2x_3,x_3x_4,\ldots,x_{m-1}x_m,x_m)Q^{(m)}_{n-k,132}(x_2,\ldots,x_m),\ \ \\
	\label{132m2} Q^{(m)}_{132}(t,x_2,\ldots,x_m)&\hspace*{-3mm}=&\hspace*{-3mm}1+tQ^{(m)}_{132}(tx_2,x_2x_3,x_3x_4,\ldots,x_{m-1}x_m,x_m)Q^{(m)}_{132}(t,x_2,\ldots,x_m).
	\end{eqnarray}
\end{theorem}
\begin{proof}
	We shall consider the distribution of the pattern $\ga_s=12\cdots s$ for $s=2,\ldots,m$ in $\Snn(132)$.
	
	Given $\sg\in\Snn(132)$ such that $\sg_k=n$, we have $A(\sg)\in\Sn{k-1}(132)$ and $B(\sg)\in\Sn{n-k}(132)$. By the recursive counting method, we have 
	\begin{eqnarray}
	\occr_{12\cdots s}(\sg)&=&\occr_{12\cdots s}(A(\sg))+\occr_{12\cdots s}(B(\sg))+\occr_{12\cdots (s-1)}(A(\sg)).
	\end{eqnarray}
	Thus,
	\begin{eqnarray}
	Q^{(m)}_{n,132}(x_2,\ldots,x_m) \hspace*{-3mm}
	&=&  \sum_{\sg\in\Sn{n}(132)} \prod_{i=2}^{m}x_i^{\occr_{12\cdots i}(\sg)} \nonumber\\
	&=& \sum_{k=1}^{n} \sum_{\pi\in\Sn{k-1}(132)} \sum_{\tau\in\Sn{n-k}(132)} 
	\prod_{i=2}^{m}x_i^{\occr_{12\cdots i}(\pi)}\cdot
	x_i^{\occr_{12\cdots i}(\tau)}\cdot
	x_i^{\occr_{12\cdots (i-1)}(\pi)}
	\nonumber\\
	&=&\hspace*{-2mm}\sum_{k=1}^{n} x_2^{k-1}Q^{(m)}_{k-1,132}(x_2x_3,x_3x_4,\ldots,x_{m-1}x_m,x_m) Q^{(m)}_{n-k,132}(x_2,\ldots,x_m),\ \ 
	\end{eqnarray}
	which proves equation (\ref{132m1}), and equation (\ref{132m2}) follows immediately from equation (\ref{132m1}).
\end{proof}

This theorem can be seen as a generalization of \tref{Q12} of \cite{FH}.

\subsection{The distribution of patterns of length four in $\Snn(132)$}
Let $\Ga_4=\{1234, 2134, 2314, 2341, 3124, 3214, 3241, 3412, 
3421, 4123, 4213, 4231, 4312, 4321\}$ be the set of patterns in $\Sn{4}(132)$. 
By \tref{an}, there are 10 of the 14 patterns in $\Ga_4$ satisfying good recursions. To track all the 14 patterns in $\Sn{4}(132)$, we shall refine the generating function $Q_n$ by the number of coinversions and define
\begin{equation}
Q_{n,i}(x_1,\ldots,x_{19}): =
Q_{n,132}^{\Ga_2\cup\Ga_3\cup\Ga_4}(x,1,x_1,\ldots,x_{19})\big|_{x^i},
\end{equation}
then
\begin{equation}\label{allS4}
Q_{n,132}^{\Ga_2\cup\Ga_3\cup\Ga_4}(x_1,\ldots,x_{21})=\sum_{i=0}^{\binom{n}{2}}x_1^i x_2^{\binom{n}{2}-i} Q_{n,i}(x_3,\ldots,x_{21}).
\end{equation}

We let $x_i$ track the occurrences of length three patterns and $y_i$ track the occurrences of length four patterns, then we have the following theorem which gives recursive relations for the generating function 
$Q_{n,i}(x_1,\ldots,x_5,y_1,\ldots,y_{14})$.

\ 
\begin{theorem}\label{theorem:S4Sn132}The function $Q_{n,i}(x_1,\ldots,x_5,y_1,\ldots,y_{14})$ satisfies the following recursion,
	\begin{eqnarray}\label{eqnS4}
	&&Q_{0,0}(x_1,\ldots,x_5,y_1,\ldots,y_{14}) =1,\\
	&&Q_{n,i}(x_1,\ldots,x_5,y_1,\ldots,y_{14}) =0\mbox{ for }i<0\mbox{ or }i>\binom{n}{2},\mbox{\ and\ \ \ \ \ \ \ \ \ \ \ \ \ \ \ \ \ \ \ \ \ \ \ \ \ \ \ \ \ \ \ \ \ \ \ \ \ \ \ \ \ \ \ }
	\end{eqnarray}
	%\begin{multline}
	%\sum_{k=1}^{n} \sum_{j=0}^{i+1-k} 
	%x_{1}^j x_{2}^{(k - 1) (k - 2)/2 - j} x_{3}^{(n - k) (k + j - 1)} x_{4}^{k (i + 1 - k - j)} x_{5}^{(n - k) ((k - 1) (k - 2)/2 - j) + k ((n - k) (n - k - 1)/2 + k + j - i - 1)}\\ y_{4}^{j (n - k)} y_{7}^{((k - 1) (k - 2)/2 - j) (n - k)} y_{8}^{(j + k - 1) (i + 1 - k - j)} y_{9}^{(j + k - 1) ((n - k) (n - k - 1)/2 + k + j - i - 1)}\\ y_{13}^{((k - 1) (k - 2)/2 - j) (i + 1 - k - j)} y_{14}^{((k - 1) (k - 2)/2 - j) ((n - k) (n - k - 1)/2 + k + j - i - 1)) }\\
	%Q_{k-1,j}(x_{1} y_{1} y_{4}^{n - k}, x_{2} y_{2} y_{7}^{n - k}, x_{3} y_{3} y_{9}^{n - k}, x_{4} y_{5} y_{12}^{n - k}, x_{5} y_{6} y_{14}^{n - k}, y_{1}, \ldots, y_{14}) \\
	%Q_{n - k, i + 1 - k - j}(x_{1} y_{10}^k, x_{2} y_{11}^k, x_{3} y_{12}^k, x_{4} y_{13}^k, x_{5} y_{14}^k, y_{1}, \ldots , y_{14})
	%\end{multline}
	\vspace*{-10mm}
	\begin{multline}
	\ \ \ Q_{n,i}(x_1,\ldots,x_5,y_1,\ldots,y_{14}) =\\
	\sum_{k=1}^{n} \sum_{j=0}^{i+1-k} 
	x_{1}^j 
	x_{2}^{\binom{k - 1}{2}- j} 
	x_{3}^{(n - k) (k + j - 1)} 
	x_{4}^{k (i + 1 - k - j)} 
	x_{5}^{(n - k) \left(\binom{k - 1}{2} - j\right) + k \left(\binom{n - k}{2} + k + j - i - 1\right)}
	y_{4}^{j (n - k)}\\\cdot
	y_{7}^{\left(\binom{k - 1}{2} - j\right) (n - k)} 
	y_{8}^{(j + k - 1) (i + 1 - k - j)} 
	y_{9}^{(j + k - 1) \left(\binom{n - k}{2} + k + j - i - 1\right)}
	y_{13}^{\left(\binom{k - 1}{2} - j\right) (i + 1 - k - j)} \\\cdot
	y_{14}^{\left(\binom{k - 1}{2} - j\right) \left(\binom{n - k}{2} + k + j - i - 1\right)) }\\\cdot
	Q_{k-1,j}(x_{1} y_{1} y_{4}^{n - k}, x_{2} y_{2} y_{7}^{n - k}, x_{3} y_{3} y_{9}^{n - k}, x_{4} y_{5} y_{12}^{n - k}, x_{5} y_{6} y_{14}^{n - k}, y_{1}, \ldots, y_{14}) \\\cdot
	Q_{n - k, i + 1 - k - j}(x_{1} y_{10}^k, x_{2} y_{11}^k, x_{3} y_{12}^k, x_{4} y_{13}^k, x_{5} y_{14}^k, y_{1}, \ldots , y_{14}).
	\end{multline}
\end{theorem}
\begin{proof}
	We shall count the number of occurrences of the 19 patterns of length three or four in $\Snn(132)$ using the recursive counting method.
	Let $\sg\in\Snn(132)$ such that $\sg_k=n$, $\occr_{12}(\sg)=i$ and $\occr_{21}(\sg)=j$, we have $A(\sg)\in\Sn{k-1}(132)$, $B(\sg)\in\Sn{n-k}(132)$, and $j=\binom{n}{2}-i$. We shall abbreviate $\occr_{\ga}(A(\sg))$, $\occr_{\ga}(B(\sg))$ and $\occr_{\ga}(\Abar(\sg))$ to $A_\ga$, $B_\ga$ and $\Abar_\ga$ in this proof.
	
	Similar to \tref{S3Sn132}, the formulas for occurrences of 123, 213, 231, 312, 321 are given by equations (\ref{o123}), (\ref{o213}), (\ref{o231}), (\ref{o312}) and (\ref{o321}). Then we shall look at the 14 patterns of length four. Case (a) of the recursive counting method implies that
	\begin{equation}\label{S4a}
	\occr_{\pi_1\pi_2\pi_3 4}(\sg)=A_{\pi_1\pi_2\pi_3 4}+B_{\pi_1\pi_2\pi_3 4}+A_{\pi_1\pi_2\pi_3}
	\end{equation}
	for any $\pi_1\pi_2\pi_3\in\Sn{3}(132)$. For patterns not ending with $4$, it follows from the recursive counting method that
	\begin{eqnarray}
	\label{o2341}
	\occr_{2341}(\sg)&=&A_{2341}+B_{2341}
	+A_{12}\cdot B_1+A_{123}\cdot B_1
	\nonumber\\&=&A_{2341}+B_{2341}+j(n-k)+(n-k)A_{123},
	\\\label{o3241}
	\occr_{3241}(\sg)&=&A_{3241}+B_{3241}
	+A_{21}\cdot B_1+A_{213}\cdot B_1\nonumber\\
	&=&A_{3241}+B_{3241}
	+\left(\binom{k-1}{2}-j\right)(n-k)+(n-k)A_{213},
	\\\label{o3412}
	\occr_{3412}(\sg)&=&A_{3412}+B_{3412}
	+A_{12}B_{12}+A_{1}B_{12}\nonumber\\
	&=&A_{3412}+B_{3412}+(j+k-1)(i+1-k-j),
	\\\label{o3421}%y9
	\occr_{3421}(\sg)&=&A_{3421}+B_{3421}+A_{12}B_{21}
	+A_{1} B_{21}+A_{231} B_{1}
	\nonumber\\
	&=&A_{3421}+B_{3421}
	+(n-k)A_{231}+(j + k - 1) \left(\binom{n - k}{2} + k + j - i - 1\right),
	\\\label{o4123}
	\occr_{4123}(\sg)&=&A_{4123}+B_{4123}+\Abar_{1} B_{123}
	=A_{4123}+B_{4123}+k B_{123},
	\\\label{o4213}
	\occr_{4213}(\sg)&=&A_{4213}+B_{4213}+\Abar_{1} B_{213}
	=A_{4213}+B_{4213}+k B_{213},
	\\\label{o4231}%y12
	\occr_{4231}(\sg)&=&A_{4231}+B_{4231}
	+\Abar_{1} B_{231}+A_{312} B_{1}
	\nonumber\\&=&A_{4231}+B_{4231}+k B_{231}+(n-k)A_{312},
	\\\label{o4312}
	\occr_{4312}(\sg)&=&A_{4312}+B_{4312}
	+\Abar_{1} B_{312}+A_{21} B_{12}
	\nonumber\\
	&=&A_{4312}+B_{4312}+k B_{312}
	+\left(\binom{k - 1}{2} - j\right) (i + 1 - k - j),
	\\\label{o4321}
	\occr_{4321}(\sg)&=&A_{4321}+B_{4321}+\Abar_{1} B_{321}\
	+A_{21} B_{21}+A_{321} B_{1}
	\nonumber\\
	&=&A_{4321}+B_{4321}+ k B_{321} 
	\nonumber\\&&+(n-k)A_{321}+\left(\binom{k - 1}{2} - j\right) \left(\binom{n - k}{2} + k + j - i - 1\right).
	\end{eqnarray}
	Then, one can arrange the generating function $Q_{n,i}(x_1,\ldots,x_5,y_1,\ldots,y_{14})$ in a similar way to equation (\ref{S3arrange}) to prove the recursion.
\end{proof}

We can compute the polynomial $Q_{n,132}^{\Ga_2\cup\Ga_3\cup\Ga_4}(x_1,\ldots,x_{21})
= \sum_{i=0}^{\binom{n}{2}}x_1^i x_2^{\binom{n}{2}-i} Q_{n,i}(x_3,\ldots,x_{21})$ efficiently by  Mathematica as follows.
\begin{multline}
Q_{132}^{\Ga_2\cup\Ga_3\cup\Ga_4}(t,x_1,\ldots,x_{21})=1+t+t^2(x_1+x_2)+t^3(x_1^3 x_3+x_1^2 x_2 x_4+x_1 x_2^2 x_5+x_1 x_2^2 x_6+x_2^3 x_7)\\
+t^4\left(x_{1}^4 x_{10} x_{2}^2 x_{3} x_{4}^2 x_{5} + x_{1}^3 x_{11} x_{2}^3 x_{3} x_{5}^3 + 
x_{1}^4 x_{12} x_{2}^2 x_{3} x_{4}^2 x_{6} + x_{1}^2 x_{15} x_{2}^4 x_{5}^2 x_{6}^2 + 
x_{1}^3 x_{17} x_{2}^3 x_{3} x_{6}^3\right.\\ + x_{1}^3 x_{13} x_{2}^3 x_{4}^3 x_{7} + 
x_{1}^2 x_{14} x_{2}^4 x_{4} x_{5}^2 x_{7} + x_{1}^2 x_{18} x_{2}^4 x_{4} x_{6}^2 x_{7} + 
x_{1} x_{16} x_{2}^5 x_{5}^2 x_{7}^2 + x_{1} x_{19} x_{2}^5 x_{5} x_{6} x_{7}^2 \\\left.+ 
x_{1} x_{2}^5 x_{20} x_{6}^2 x_{7}^2 + x_{2}^6 x_{21} x_{7}^4 + x_{1}^6 x_{3}^4 x_{8} + 
x_{1}^5 x_{2} x_{3}^2 x_{4}^2 x_{9}\right)+\cdots.
\end{multline}

\section{The functions $Q_{123}^{\gamma}(t,x)$}
We use the bijection $\Psi:\Sn{n}(123) \rightarrow 
\mathcal{D}_n$ of \cite{EliDeu} to study the distribution of the patterns 132 and 231 in $\Snn(123)$. When computing the distribution of 132-distribution, we prove a stronger result about $1m\cdots 2$-distribution in $\Snn(123)$. Our method does not apply for the pattern 321 due to the complexity of the structure of 123-avoiding permutations.

\subsection{The distribution of the pattern $1m\cdots 2$ in $\Snn(123)$}
Given $\sg=\sg_1\cdots\sg_n\in\Snn(123)$. If $\sg_{i_1}\cdots\sg_{i_m}$ is an occurrence of the pattern $1m\cdots2$, then the number $\sg_{i_1}$ must be a left-to-right minimum of $\sg$, otherwise there exists a number $\sg_a<\sg_{i_1}$ with index $a<i_1$, and $(a,i_1,i_m)$ is an occurrence of the pattern 123.

Under the map $\Psi:\Sn{n}(123) \rightarrow \mathcal{D}_n$, the number $\sg_{i_1}$ sits above a peak of the corresponding Dyck path $\Psi(\sg)$. Suppose that the number $\sg_{i_1}$ is on the \thn{d} diagonal of $\Psi(\sg)$, then by (3) of \lref{p2}, there are $d$ numbers to the right of $\sg_{i_1}$ that are bigger than $\sg_{i_1}$, appearing in a decreasing way. It follows immediately that there are $\binom{d}{m-1}$ occurrences of the pattern $1m\cdots2$ at the peak $\sg_{i_1}$ since any $m-1$ of the $d$ numbers to the right of $\sg_{i_1}$ that are bigger than $\sg_{i_1}$ create a pattern $1m\cdots2$ with the number $\sg_{i_1}$.

Now let $c_d(\sg)$ be the number of peaks that are on the \thn{d} diagonal of $\Psi(\sg)$, then 
\begin{equation}\label{ddd}
\occr_{1m\cdots2}(\sg)=\sum_{d\geq 0} c_d(\sg) \binom{d}{m-1}.
\end{equation}
We also let $c_d(P)$ denote the number of peaks that are on the \thn{d} diagonal of a path $P$.

We  define 
\begin{eqnarray}
Q^{(m)}_{n,123}(s,x_2,x_3,\ldots,x_m)&: = &\hspace*{-3mm} \sum_{\sg\in\Snn(123)}s^{\LRmin(\sg)}x_2^{\occr_{12}(\sg)}x_3^{\occr_{132}(\sg)}\cdots x_m^{\occr_{1m(m-1)\cdots 2}(\sg)}\mbox{ \ and \ \ \ \ }\\
Q^{(m)}_{123}(t,s,x_2,x_3,\ldots,x_m)&: = &\sum_{n\geq 0}t^nQ_{n,123}(s,x_2,x_3,\ldots,x_m),
\end{eqnarray}
then we have the following theorem.

\ 
\begin{theorem}\label{theorem:Qm123}
	Given $n\geq 0$ and $m\geq 2$, 
	\begin{multline}
	\label{Qm1231}Q^{(m)}_{n,123}(s,x_2,\ldots,x_m)=sQ^{(m)}_{n-1,123}(s,x_2,\ldots,x_m)\\
	+\sum_{k=2}^{n} Q^{(m)}_{k-1,123}(sx_2,x_2x_3,x_3x_4,\ldots,x_{m-1}x_m,x_m)Q^{(m)}_{n-k,123}(s,x_2,\ldots,x_m),
	\end{multline}
	and
	\begin{multline}
	\label{Qm1232}Q^{(m)}_{123}(t,s,x_2,\ldots,x_m)=1+t(s-1)Q^{(m)}_{123}(t,s,x_2,\ldots,x_m)\\
	+t Q^{(m)}_{123}(t,sx_2,x_2x_3,x_3x_4,\ldots,x_{m-1}x_m,x_m)Q^{(m)}_{123}(t,s,x_2,\ldots,x_m).
	\end{multline}
\end{theorem}
\begin{proof}
	We enumerate the pattern occurrences using the Dyck path bijection $\Psi$. Given any Dyck path $P$, we can break the path at the first return to write $P$ as $DP_1RP_2$, where $P_1$ is the path after the first $D$ step before the last $R$ step before the first return, and $P_2$ is the path after the first return.
	
	Let $k=\ret(P)$. By equation (\ref{ddd}), we have 
	\begin{eqnarray*}
		&&Q^{(m)}_{n,123}(s,x_2,\ldots,x_m)\nonumber\\
		&=&  \sum_{\sg\in\Snn(123)}s^{\sum_{d\geq 0} c_d(\sg) \binom{d}{0}} x_2^{\sum_{d\geq 0} c_d(\sg) \binom{d}{1}}x_3^{\sum_{d\geq 0} c_d(\sg) \binom{d}{2}}\cdots x_m^{\sum_{d\geq 0} c_d(\sg) \binom{d}{m-1}}   \nonumber\\
		&=&  \sum_{P\in\Dnn}s^{\sum_{d\geq 0} c_d(P) \binom{d}{0}} x_2^{\sum_{d\geq 0} c_d(P) \binom{d}{1}}x_3^{\sum_{d\geq 0} c_d(P) \binom{d}{2}}\cdots x_m^{\sum_{d\geq 0} c_d(P) \binom{d}{m-1}}  \nonumber\\
		&=&  s \sum_{P_2\in\Dn{n-1}}s^{\sum_{d\geq 0} c_d(P_2) \binom{d}{0}} x_2^{\sum_{d\geq 0} c_d(P_2) \binom{d}{1}}\cdots x_m^{\sum_{d\geq 0} c_d(P_2) \binom{d}{m-1}}
		\nonumber\\&&+\sum_{k=2}^{n}\sum_{P_1\in\Dn{k-1}}s^{\sum_{d\geq 0} c_d(P_1) \binom{d+1}{0}} x_2^{\sum_{d\geq 0} c_d(P_1) \binom{d+1}{1}} \cdots x_m^{\sum_{d\geq 0} c_d(P_1) \binom{d+1}{m-1}} 
		\nonumber\\&&
		\cdot \sum_{P_2\in\Dn{n-k}} s^{\sum_{d\geq 0} c_d(P_2) \binom{d}{0}} x_2^{\sum_{d\geq 0} c_d(P_2) \binom{d}{1}} \cdots x_m^{\sum_{d\geq 0} c_d(P_2) \binom{d}{m-1}} \nonumber\\
		&=&  sQ^{(m)}_{n-1,123}(s,x_2,\ldots,x_m)
		\nonumber\\&&+\sum_{k=2}^{n}\sum_{P_1\in\Dn{k-1}}s^{\sum_{d\geq 0} c_d(P_1) \binom{d}{0}} x_2^{\sum_{d\geq 0} c_d(P_1) \left(\binom{d}{0}+\binom{d}{1}\right)} \cdots x_m^{\sum_{d\geq 0} c_d(P_1) \left(\binom{d}{m-2}+\binom{d}{m-1}\right)} 
		\nonumber\\&&
		\cdot \sum_{P_2\in\Dn{n-k}} s^{\sum_{d\geq 0} c_d(P_2) \binom{d}{0}} x_2^{\sum_{d\geq 0} c_d(P_2) \binom{d}{1}} \cdots x_m^{\sum_{d\geq 0} c_d(P_2) \binom{d}{m-1}} \nonumber\\
		&=&sQ^{(m)}_{n-1,123}(s,x_2,\ldots,x_m)\nonumber\\
		&&+\sum_{k=2}^{n} Q^{(m)}_{k-1,123}(sx_2,x_2x_3,x_3x_4,\ldots,x_{m-1}x_m,x_m)Q^{(m)}_{n-k,123}(s,x_2,\ldots,x_m).
	\end{eqnarray*}
	Equation (\ref{Qm1232}) follows immediately from equation (\ref{Qm1231}).
\end{proof}

Evaluating $m$ at $3$ gives the following corollary for the distribution of the coinversion and the $\occr_{132}$ statistics.

\ 
\begin{corollary}$Q_{n,123}^{(3)}(s,q,x)=\sum_{\sg\in\Snn(123)}s^{\LRmin(\sg)}q^{\coinv(\sg)}x^{\occr_{132}(\sg)}$ satisfies
	\begin{eqnarray}
	Q_{0,123}^{(3)}(s,q,x)=1, \ \ \ Q_{n,123}^{(3)}(s,q,x)=s Q^{(3)}_{n-1,123}+\sum_{k=2}^{n} Q^{(3)}_{k-1,123}(sq,qx,x) Q^{(3)}_{n-k,123}(s,q,x).
	\end{eqnarray}
	Further, 
	\begin{equation}
	Q^{(3)}_{123}(t,s,q,x) = 1+t(s-1)Q^{(3)}_{123}(t,s,q,x)\nonumber\\
	+t Q^{(3)}_{123}(t,sq,qx,x)Q^{(3)}_{123}(t,s,q,x).
	\end{equation}
\end{corollary}

Then we can use the recursive formula to compute $Q^{\{12,132\}}_{123}(t,q,x)=\sum_{n\geq 0}t^n Q_{n,123}^{(3)}(1,q,x)$:
\begin{multline}
Q^{\{12,132\}}_{123}(t,q,x)=1 +t + t^2 (1 + q)  + t^3 (1 + 2 q + q^2 + q^2 x) + 
t^4 (1 + 3 q + 3 q^2 + q^3 + 2 q^2 x + 2 q^3 x\\
 + q^4 x^2 + q^3 x^3) +
t^5 (1 + 4 q + 6 q^2 + 4 q^3 + q^4 + 3 q^2 x + 6 q^3 x + 3 q^4 x + 
3 q^4 x^2 + 2 q^5 x^2 + 2 q^3 x^3 + 2 q^4 x^3\\ + q^6 x^3 + 
2 q^5 x^4 + q^4 x^6 + q^6 x^6)+\cdots.
\end{multline}

By looking at the coefficients of the generating functions, we observe a property that
\begin{equation}\label{132123}
| \{\sg\in\Snn(132):\occr_{12\cdots j}(\sg)=i\}|=|\{\sg\in\Snn(123):\occr_{1j(j-1)\cdots 2}(\sg)=i\}|\textnormal{ \ \ for all }i<j.
\end{equation}
Let $[x^i]_Q$ denote the coefficient of $x^i$ in the function $Q$, then the property above is equivalent to the following theorem.

\ 
\begin{theorem}\label{theorem:132123}
	For any nonnegative integers $i<j$, 
	\begin{equation}\label{coeff}
	[t^n x^i]_{Q_{132}^{1\cdots j}(t,x)}=[t^n x^i]_{Q_{123}^{1j\cdots 2}(t,x)}.
	\end{equation}
\end{theorem}
\begin{proof}
	One can use the recursive equations (\ref{132m1}) and (\ref{Qm1231}) to prove the theorem using induction. Here we shall give an alternative combinatorial proof of the theorem using the Dyck path bijections $\Phi$ and $\Psi$.
	
	By equation (\ref{ddd}), $\sg\in\Snn(123)$ has $i$ occurrences of the pattern $1j(j-1)\cdots 2$ where $j>i$ if and only if the corresponding Dyck path $\Psi(\sg)$ has $i$ peaks on the $j-1^{\textnormal{st}}$ diagonal and no peaks on the \thn{k} diagonal for all $k\geq j$.
	
	On the other hand, let $\pi\in\Snn(132)$ and $\occr_{12\cdots j}(\pi)=i$. The corresponding Dyck path $\Phi(\pi)$ has no peaks on the \thn{k} diagonal for all $k\geq j$. Otherwise, if $\pi_i$ is on the \thn{k} diagonal for some $k\geq j$, there are $k$ numbers to the right of $\pi_i$ that are greater than $\pi_i$, forming a length $k+1$ increasing subsequence together with $\pi_i$. There are $\binom{k+1}{j}>i$ occurrences of the pattern $12\cdots j$ in this subsequence, which leads to a contradiction.
	
	Further, if $\pi_{\ell_1}\cdots\pi_{\ell_j}$ is an occurrence of the pattern $12\cdots j$, then $\pi_{\ell_1}$ must be a peak on the $j-1^{\textnormal{st}}$ diagonal, otherwise any peak to the left of $\pi_{\ell_1}$ that is smaller than $\pi_{\ell_1}$ is on the \thn{k} diagonal for some $k\geq j$ by \lref{p1} (c), contradicting with the statement that $\Phi(\pi)$ has no peaks on the \thn{k} diagonal for all $k\geq j$.
	
	Thus, $\pi$ has $i$ occurrences of the pattern $1\cdots j$ where $j>i$ if and only if the corresponding Dyck path $\Phi(\pi)$ has $i$ peaks on the $j-1^{\textnormal{st}}$ diagonal and no peaks on the \thn{k} diagonal for all $k\geq j$, which proves equations (\ref{132123}) and (\ref{coeff}).
\end{proof}

\subsection{The distribution of the pattern $231$ in $\Snn(123)$}
We give recursive formulas for the generating function of $\Snn(123)$ tracking the number of occurrences of the pattern $231$ by refining function $Q_n$ by the number of left-to-right minima. Given $\sg\in\Snn(123)$, we let $\linv(\sg)$ be the number of pairs $(i,j)$ such that $\sg_i$ is a left-to-right minimum, $\sg_j$ is not a left-to-right minimum and $\sg_i>\sg_j$. For a Dyck path $P$, we also let $\linv(P)=\linv(\Psi^{-1}(\sg))$.

Next, we define
\begin{eqnarray}
D_n (s,q,x,y)&: =& \sum_{\sg\in\Snn(123)} s^{\LRmin(\sg)} q^{\occr_{12}(\sg)}  x^{\linv(\sg)} y^{\occr_{231}(\sg)}\mbox{ and }\\
D_{n,k} (q,x,y)&: =& \sum_{\sg\in\Snn(123), \LRmin(\sg)=k} q^{\occr_{12}(\sg)}  x^{\linv(\sg)} y^{\occr_{231}(\sg)},
\end{eqnarray}
then $D_n (s,q,x,y)=\sum_{k=1}^{n}s^k D_{n,k} (q,x,y)$, and we have the following theorem for $D_n (s,q,x,y)$.

\ 
\begin{theorem}\label{theorem:123231}
	$D_{0} (s,q,x,y) =D_{0,0}(q,x,y) =1$. For any $n,k\geq 1$, 
	\begin{equation}
	D_{n,1} (q,x,y) = q^{n-1},\ D_{n,n}(q,x,y)=1,\ D_{n,k}(q,x,y)=0 \mbox{ for }k>n,
	\end{equation}
	and
	\begin{multline}\label{123231}
	D_{n,k}(q,x,y)= x^{n-k} D_{n-1,k-1}(q,x,y) +q^{k} D_{n-1,k}(q,xy,y)\\
	+\sum_{i=2}^{n-1}\sum_{j=\max(1,k+i-n)}^{\min(i-1,k-1)} q^j x^{j(n-i-k+j)} y^{j(n-i)} D_{i-1,j}(qy^{n-i},xy,y) D_{n-i,k-j}(q,x,y).
	\end{multline}
\end{theorem}
\begin{proof}
	Given $\sg\in\Snn(123)$ such that $\LRmin(\sg)=k$, we let $P=\Psi(\sg)$ be the corresponding Dyck path which has $k$ peaks by \lref{p2}. Suppose that $\ret(\sg)=i$, then similar to \tref{Qm123}, we can write $P=DP_1RP_2$, where $P_1$ is a Dyck path of size $i-1$ and $P_2$ is a Dyck path of size $n-i$. 
	
	If $i=1$, then $P_1$ is empty, and $DP_1R$ is a peak on the main diagonal, thus $P_2$ should be a Dyck path of size $n-1$ with $k-1$ peaks. There are $n-k$ extra linvs between the first peak and the $n-k$ non-peaks in $P_2$, thus the contribution of this case is $x^{n-k} D_{n-1,k-1}(q,x,y)$.
	
	If $i=n$, then $P_2$ is empty, and $P=DP_1R$. $P_1$ should be a Dyck path of size $n-1$ with $k$ peaks. There are $k$ more inversions of $\Psi^{-1}(DP_1R)$ than $\Psi^{-1}(P_1)$, and there are $\linv(P_1)$ more occurrences of the pattern $231$ in  $\Psi^{-1}(DP_1R)$ than $\Psi^{-1}(P_1)$,  thus the contribution of this case is  $q^{k} D_{n-1,k}(q,xy,y)$.
	
	If $1<i<n$, then both $P_1$ and $P_2$ are not empty. Suppose that there are $j$ peaks in $P_1$, then there are $k-j$ peaks in $P_2$. Other than the statistics counted inside $P_1$ and $P_2$, there are $j$ more inversions of $\Psi^{-1}(DP_1R)$ than $\Psi^{-1}(P_1)$, $j(n-i-k+j)$ extra linvs between $P_1$ and $P_2$, and $j(n-i)+(n-i)\occr_{12}(\Psi^{-1}(P_1))+\linv(P_1)$ extra occurrences of the pattern $231$, thus the contribution of this case is $q^j x^{j(n-i-k+j)} y^{j(n-i)} D_{i-1,j}(qy^{n-i},xy,y) D_{n-i,k-j}(q,x,y)$.
	
	Summing over all the cases gives (\ref{123231}).
\end{proof}

Then we can compute $Q^{\{12,231\}}_{123}(t,q,x)=\sum_{n\geq 0} t^n D_n(1,q,1,x)$ using the recursive formula to get
\begin{multline}
Q^{\{12,231\}}_{123}(t,q,x)=1 +t + (1 + q) t^2 + t^3 (1 + q + 2 q^2 + q x) + 
t^4 (1 + q + 2 q^2 + 2 q^3 + q^4 + q x + 2 q^3 x + q x^2 \\
+ 
3 q^2 x^2)+ 
t^5 (1 + q + 2 q^2 + 2 q^3 + 3 q^4 + 2 q^6 + q x + 2 q^3 x + 
2 q^5 x + q x^2 + 3 q^2 x^2 + 5 q^4 x^2 + 2 q^5 x^2 \\+ q x^3 + 
q^2 x^3 + 4 q^3 x^3 + q^4 x^3 + 3 q^2 x^4 + 4 q^3 x^4 + q^4 x^4)+\cdots.
\end{multline}

\section{Applications in pattern popularity}
Let $S$ be a set of permutations and $\ga$ be a permutation pattern. The \emph{popularity} of $\ga$ in $S$, $f_{S}(\ga)$, is defined by
\begin{equation}
f_{S}(\ga):=\sum_{\sg\in S} \occr(\ga).
\end{equation}
Let
\begin{eqnarray}
F_{\ga}(t)&:=&\sum_{n\geq 0} f_{\Snn(132)}(\ga) t^n\mbox{ \ \ and }\\
G_{\ga}(t)&:=&\sum_{n\geq 0} f_{\Snn(123)}(\ga) t^n,
\end{eqnarray}
\cite{Bona} and \cite{Hom} studied the popularity of length two or three patterns in $\Snn{(132)}$ and $\Snn(123)$. 

\ 
\begin{theorem}[B\'{o}na and Homberger]
	Let $C(t):=\sum_{n\geq 0}C_n t^n$ be the generating function of Catalan numbers. Then
	\begin{eqnarray}
	F_{12}(t)&=&\frac{t^2 C^3(t)}{\left(1-2tC(t)\right)^2},\\
	G_{12}(t)&=&\frac{t C^2(t)}{1-2tC(t)}.
	\end{eqnarray}
\end{theorem}

In this section, we shall apply our results in Section 3 and Section 4 in pattern popularity problems, and we have the following theorem.

\ 
\begin{theorem}
	Let $m>2$ be an integer. Then
	\begin{eqnarray}
	\label{pop1}
	F_{12\cdots m}(t)&=& \frac{t C(t) F_{12\cdots (m-1)}(t)}{1-2tC(t)}, \mbox{ \ \ and}\\
	\label{pop2}
	G_{1m\cdots 2}(t)&=& \frac{t C(t) G_{1(m-1)\cdots 2}(t)}{1-2tC(t)}.
	\end{eqnarray}
\end{theorem}
\begin{proof}
	Equation (\ref{pop1}) is a consequence of equation (\ref{132m2}).
	To be more explicitly, it is a fact that
	\begin{equation}
	Q^{(m)}_{132}(t,x_2,\ldots,x_m)\big|_{x_2=\cdots=x_m=1}=C(t).
	\end{equation}
	It follows from the definition of $Q^{(m)}_{132}(t,x_2,\ldots,x_m)$ that
	\begin{equation}
	F_{12\cdots m}(t)=\frac{\partial Q^{(m)}_{132}(t,x_2,\ldots,x_m)}{\partial x_m}\bigg|_{x_2=\cdots=x_m=1}.
	\end{equation}
	Taking partial derivative of equation (\ref{132m2}) over $x_m$ and evaluating $x_2,\ldots,x_m$ at 1 give
	\begin{eqnarray}
	F_{12\cdots m}(t)&=&\frac{\partial Q^{(m)}_{132}(t,x_2,\ldots,x_m)}{\partial x_m}\bigg|_{x_2=\cdots=x_m=1}\nonumber\\
	&=&  t\left(\frac{\partial Q^{(m)}_{132}(t,x_2,\ldots,x_m)}{\partial x_m} Q^{(m)}_{132}(t,x_2,\ldots,x_m)\right.  \nonumber\\
	&&+  \frac{\partial Q^{(m)}_{132}(t,x_2,\ldots,x_m)}{\partial x_{m-1}} Q^{(m)}_{132}(t,x_2,\ldots,x_m)     \nonumber\\
	&&\left.+  \frac{\partial Q^{(m)}_{132}(t,x_2,\ldots,x_m)}{\partial x_m} Q^{(m)}_{132}(t,x_2,\ldots,x_m)  \right) \Bigg|_{x_2=\cdots=x_m=1}  \nonumber\\
	&=& t \left( C(t)F_{12\cdots (m-1)}(t) + 2C(t)F_{12\cdots m}(t)\right),
	\end{eqnarray}
	which implies equation (\ref{pop1}).
	
	Equation (\ref{pop2}) is a consequence of equation (\ref{Qm1232}) and can be proved in a similar way. We shall omit the proof of equation (\ref{pop2}).
\end{proof}

\section{Summary and future work}
We have obtained recursions of generating functions of $\Snn(132)$ tracking all patterns of length two, three or four. In fact, it is possible to give a recursion for the generating function $Q_{n,132}^{\ga}(t,x)$ of $\Snn(132)$ tracking the occurrences of a pattern $\ga$ of any length in $\Snn(132)$ if one does enough refinement. 

On $\Snn{(123)}$, we have only obtained recursions of generating functions tracking two patterns of length two, two patterns of length three and the special pattern $1m(m-1)\cdots 2$. The recursions about $\Snn(123)$ tend to be more complicated than those about $\Snn(132)$, and we have not succeeded in computing a recursion for the  generating function tracking the pattern $321$.

We have applied our results about classical pattern distributions in pattern popularity problems, and we have obtained nice results on pattern $1\cdots m$ popularity in $\Snn(132)$ and pattern $1m\cdots 2$ popularity in $\Snn(123)$.

We have noticed other equalities of coefficients of generating functions $Q_{132}^\gamma(t,x)$ and $Q_{123}^\gamma(t,x)$ except equation (\ref{coeff}). For example, 
the number of permutations in $\Snn(123)$ with one occurrence of the pattern $231$ is equal to the number of permutations in $\Snn(231)$ with one occurrence of the pattern $123$, which is equal to $2n-5$; the number of permutations in $\Snn(132)$ with one occurrence of the pattern $3412$ is equal to the number of permutations in $\Snn(132)$ with one occurrence of the pattern $2341$, which is one less than the \thn{2n-5} Fibonacci number. 

We have not studied sets of permutations avoiding patterns of length bigger than three. We shall study such problems in the future.

\acknowledgements
\label{sec:ack}
The first author would  like to thank the second author, Professor Jeff Remmel who passed away recently, for his mentorship, friendship and collaboration.

The authors would also like to thank 
Cheyne Homberger and Professor Brendon Rhoades for helpful discussions.

\nocite{*}
\bibliographystyle{abbrvnat}
% use the following instead if you encounter problems 
%\bibliographystyle{alpha}
\bibliography{classical}
\label{sec:biblio}

\end{document}